\documentclass[12pt,a4paper]{article}
\usepackage{graphicx}
\usepackage{amsmath}
\usepackage{amssymb}
\usepackage{amscd}
\usepackage{amsthm}
\usepackage{cleveref}
\usepackage{enumitem}
\usepackage[backend=bibtex,
    style=numeric,
    giveninits=true,
    doi=false,
    sorting=nyt,
    sortcites,
    maxnames=5]{biblatex}
\makeatletter
\renewcommand*{\@fnsymbol}[1]{\@arabic{#1}}
\makeatother

\newtheorem{theorem}{Theorem}[section]
\newtheorem{definition}[theorem]{Definition}
\newtheorem{corollary}[theorem]{Corollary}
\newtheorem{proposition}[theorem]{Proposition}
\newtheorem{lemma}[theorem]{Lemma}

\newtheorem{remark}[theorem]{Remark}

\newcommand{\pr}{\indent{\em Proof: \ }}

\newcommand{\EQ}{\begin{equation}}
\newcommand{\EN}{\end{equation}}

\newcommand{\B}{\cal{B}}

\newcommand{\F}{\mathbb{F}}

\newenvironment{demo}{\noindent {\pr}\ }{\qed}

\newcommand{\bv}{\mathbf{v}}

\newcommand{\bh}{{\bf h}}

\newcommand{\rr}{\bar{r}}

\newcommand{\br}{\mathbf{r}}

\newcommand{\bx}{\mathbf{x}}

\newcommand{\wt}{\operatorname{wt}}
\newcommand{\Aut}{\operatorname{Aut}}
\newcommand{\IA}{\operatorname{IA}}

\newcommand{\Rank}{\operatorname{Rank}}
\newcommand{\Gcd}{\operatorname{gcd}}
\newcommand{\SL}{\operatorname{\Gamma L}}
\newcommand{\GL}{\operatorname{GL}}

\title{On Completely Regular Codes\footnote{This work has been partially
supported by the Spanish MINECO grants TIN2016-77918-P (AEI/FEDER, UE) and
MTM2015-69138-REDT, and by the Catalan AGAUR grant 2014SGR-691 and
also by the Russian fund of fundamental researches (15-01-08051).}}
\author{J. Borges, J. Rif\`{a}\footnote{email:~josep.rifa@autonoma.edu, joaquim.borges@autonoma.edu} \\
Dept. of Information and Communications Engineering,\\
 Universitat Aut\`{o}noma de Barcelona,\\
\and V. A. Zinoviev\footnote{e-mail:\, zinov@iitp.ru}\\A.A. Kharkevich Inst. for Problems of Information
Transmission,\\ Russian Academy of Sciences}
\date{}
\begin{document}
\maketitle

\begin{abstract}
This work is a survey on completely regular codes. Known
properties, relations with other combinatorial structures and
constructions are stated. The existence problem is also discussed
and known results for some particular cases are established.
In particular, we present a few new results on completely regular codes with $\rho=2$ and on extended
completely regular codes.

\end{abstract}
\newpage
\tableofcontents
\newpage
\section{Introduction}

In 1973, completely regular codes in Hamming metric were
introduced by Delsarte \cite{Dels}. Such codes have combinatorial
properties generalizing those of perfect codes. The class of
completely regular codes includes perfect and extended
perfect codes \cite{SZZ}, but also uniformly packed codes
\cite{GvT,SZZ,vTi}, codes obtained by their extensions \cite{SZZ},
and completely transitive codes \cite{GP2,Giu1,Giu2,Sole}. Known
completely regular codes are, for example, Hamming, Golay,
Preparata, some BCH codes with $d=5$ and some Hadamard codes. The
combinatorial properties of completely regular codes allow to
establish different relations with other combinatorial structures
such as distance-regular graphs, association schemes and designs.
A comprehensive text about these relations is a monograph of
Brouwer, Cohen, and Neumaier \cite[Ch. 11]{BCN} complemented
in a survey of van Dam, Koolen, and Tanaka \cite{Dam2}. A table
of possible parameters of
completely regular codes of finite lengths and their intersection
arrays can be found in \cite{koo}.

It is known that completely regular codes exist for arbitrary
large covering radius (see, for example, the direct construction
of Sol\'{e} \cite{Sole}). However, there are no known nontrivial completely
regular codes with large error-correcting capability. Concretely,
there are no known completely regular codes with minimum distance
$d>8$ and more than two codewords. In 1973, it has been proven the
nonexistence of unknown nontrivial perfect codes over finite
fields independently by Tiet\"{a}v\"{a}inen \cite{Tiet} and  by
Leontiev and Zinoviev \cite{LZ}. The same result was obtained in
1975 for quasi-perfect uniformly packed codes by Goethals and van
Tilborg \cite{vTi,GvT} (infinite families were ruled out earlier in \cite{SZZ}).
For the particular case of binary linear
completely transitive codes, Borges, Rif\`{a} and Zinoviev also
proved the nonexistence for $d>8$ and more than two codewords in
2001 \cite{Trans2}. Therefore, a natural conjecture seems to be
the nonexistence of nontrivial completely regular codes for $d>8$.
In 1992, Neumaier \cite{Neum} conjectured that the only completely
regular code containing more than two codewords with $d\geq 8$ is
the extended binary Golay code. However, Borges, Rif\`{a} and
Zinoviev found a counter example of Neumaier's conjecture
\cite{Bor3}. More precisely, they proved than the even half of the
binary Golay code is also completely regular. However, the
existence of unknown nontrivial completely regular codes for $d \geq 8$ remains
an open question.

An interesting subclass of completely regular codes are the completely transitive codes, first introduced by Sol\'{e} \cite{Sole} and later extended over a Hamming graph by Giudici and Praeger \cite{Giu2}. More recently, Koolen, Lee, Martin and Tanaka studied and classified the class of arithmetic completely regular codes \cite{koo2}.

In the next section, we give the main definitions, equivalences and preliminary results. In \Cref{Lloyd},
we give some necessary conditions for the existence of completely regular codes.
In \Cref{CTC}, we study completely transitive codes, a particular case of completely
regular codes. In \Cref{ECR}, we consider extension of completely regular codes and
give some new results.
In \Cref{Cons}, we give different constructions of completely
regular codes.

\section{Preliminary results}\label{defs}

\subsection{Completely regular codes and related classes}

We consider codes over finite fields $\F_q=GF(q)$, $q$ being a prime power, and
the Hamming metric. For codes over rings, the Lee metric is often used. In many
cases, such codes can be viewed as binary codes under Gray maps, hence we can
consider them as binary codes and the distance is again the Hamming distance.
As usual, for a code $C\subset\F_q^n$, we denote by $n$, $d$, $e$ and $\rho$,
the length, the minimum distance, the packing radius (or error-correcting
capability), and the covering radius of $C$. If $C$ is linear, then $k$ denotes
its dimension. We shall use the standard notation $(n,M,d)_q$ to denote a $q$-ary
code of length $n$, size $M$ and minimum distance $d$. If the code is linear, then
we indicate the dimension of the code instead of the size, and the notation is $[n,k,d]_q$. If we want to specify also the covering radius of the code, then we write $(n,M,d;\rho)_q$ for a nonlinear code, or $[n,k,d;\rho]_q$ for a linear code. For the binary case ($q=2$), we usually omit the subscript $q$.
Unless stated otherwise, we assume that $C$ is a distance invariant code \cite{MacW}
containing the zero vector.

For a binary code $C$, $\Aut(C)$ denotes the group of automorphisms of $C$, i.e.
the set of coordinate permutations that fixes $C$ set-wise.

We call {\em trivial} a code $C$ with size $|C|\leq 2$ or $C=\F_q^n$.

Let $C\subset \F_q^n$ be a code. Given a vector $x\in \F_q^n$, we denote by $B_{x,i}$
the number of codewords at distance $i$ from $x$. The {\em outer distribution matrix}
of $C$ is a $q^n\times (n+1)$ matrix $B$, with entries
$$
B_{x,i}=\mid\{v\in C \mid d(x,v)=i\}\mid,
$$
hence the row $B_x$ is the weight distribution of the translate $C+x$. Denote by $b+1$
the number of distinct rows of $B$.

Define also the sets
$$
C(i)=\{x\in\F_q^n\mid d(x,C)=i\}.
$$
The sets $C,C(1),\ldots,C(\rho)$ are called {\em subconstituents} in \cite{Neum} and
{\em cells} by some other authors. Note that $C(0)=C$ and
$C_t\neq\emptyset$ if and only if $t\leq\rho$.

\begin{definition}[\protect{\cite{GvT}}]
A code $C$ with covering radius $\rho$ is $t$-{\em regular} ($0\leq t\leq \rho$) if for
all $i=0,\ldots,\rho$, $B_{x,i}$ depends only on $i$ and on the distance $d(x,C)$ of $x$ to $C$,
for all $x$ such that $d(x,C)\leq t$.
\end{definition}

In other words, $C$ is $t$-regular if $B_x$ depends only on $d(x,C)$ for $d(x,C)\leq t$.
By definition, if $C$ is $t$-regular, then it is $j$-regular, for all $j=0,\ldots,t$.
Intuitively, $C$ is $t$-regular if we ``see" the same amount of codewords at the same
distances from any vector at distance at most $t$ from the code. For example, a
$0$-regular code is exactly a distance invariant code.

There are several equivalent definitions of completely regular code (CR code, for short).

\begin{definition}[\protect{\cite{Dels}}]\label{Def1}
A code $C$ is {\em completely regular} if it is $\rho$-regular.
\end{definition}
Clearly, the following definitions are equivalent.
\begin{enumerate}[label=(\emph{\roman*})]
\item A code $C$ is CR if for all $x\in\F_q^n$, $B_{x,i}$
depends only on $i$ and $d(x,C)$.
\item A code $C$ is CR if for all $x\in\F_q^n$, $B_x$
depends only on $d(x,C)$.
\item A code $C$ is CR if the weight distribution of
any translate $C+x$ depends only on $d(x,C)$.
\item \label{def:2.2} A code $C$ is CR if $b=\rho$.
\end{enumerate}

But it is not so straightforward the equivalence with the following definition.

\begin{definition}[\protect{\cite{Neum}}]\label{Def2}
A code $C$ is {\em completely regular} if, for all $\ell\geq 0$, every vector $x\in C(\ell)$
has the same number $c_\ell$ of neighbors in $C(\ell-1)$ and the same number $b_\ell$ of
neighbors in $C(\ell+1)$, where we set $c_0=b_\rho=0$.
\end{definition}

It is clear that the sets $C,C(1),\ldots,C(\rho)$ give a partition, called
{\em distance partition} of the space $\F_q^n$. If the condition of \Cref{Def2}
is satisfied, then the partition is called {\em equitable}. Therefore, $C$ is CR if and only if the distance partition is equitable.

For the equivalence between \Cref{Def1} and \Cref{Def2}, see \cite[Th. 4.1]{Neum}
or \cite{Bro1}. For $\ell\geq 0$, let $a_\ell=n(q-1)-b_\ell-c_\ell$.
Thus, $a_\ell$ is the number of neighbors in $C(\ell)$ of any vector in $C(\ell)$. The parameters
$a_\ell$, $b_\ell$ and $c_\ell$ are called the {\em intersection numbers} and the sequence
$\IA=\{b_0,\ldots,b_{\rho-1};c_1,\ldots,c_\rho\}$ is called the {\em intersection array}  of $C$.

Consider examples of CR codes. Recall that a
code $C$ is {\em perfect} if $\rho=e$ and {\em quasi-perfect} if $\rho=e+1$. Nontrivial
perfect codes exist only for $e\leq 3$ \cite{Tiet,LZ}. The following perfect
codes are known in $\F_q^n$:

\begin{enumerate}
\item The trivial codes $C=\{x\}$~ ($x\in\F_q^n$) and $C=\F_q^n$.
\item The binary repetition codes of odd length with $d=n$.
\item The binary Golay code with $n=23$, $k=12$, $d=7$.
\item The ternary Golay code with $n=11$, $k=6$, $d=5$.
\item The perfect codes with size $M=q^{n-m}$,
length $n=(q^m-1)/(q-1)$ and distance $d=3$, where $q$ is a prime power.
\end{enumerate}

In the last case, for every $n=(q^m-1)/(q-1)$ there is a unique linear version
which is the Hamming code.

An interesting class of codes, closely related to CR
codes, is the class of uniformly packed (UP) codes. There are three mostly
used different concepts of UP codes. The term ``uniformly packed
code" was first used in \cite{SZZ}.

\begin{definition}[\protect{\cite{SZZ}}] \label{def:2.4}
A binary quasi-perfect code $C$ with minimum distance $d=2e+1$ is
{\em uniformly packed in the narrow sense} if there exist a
natural number $\mu$, such that $B_{x,e} + B_{x,e+1} = \mu$ for
any vector $x \in C(e) \cup C(e+1)$.
\end{definition}

We emphasize that all binary perfect codes fall into this category
for $\mu = (n+1)/(e+1)$. Uniformly packed in the narrow sense
codes include: $1$-shortened binary perfect codes, Pre\-pa\-ra\-ta
codes, binary $BCH$ codes with designed distance $5$ of length
$2^{2u+1}-1$, the Hadamard code of length $11$. In \cite{SZZ} it
was proved, in terms of \Cref{Def1} \ref{def:2.2}, that UP
codes in the narrow sense and codes obtained by their extensions
(which are not uniformly packed in this sense) are CR.

UP codes in the narrow sense are a subclass of
UP codes, introduced by Goethals and van Tilborg
\cite{GvT}.

\begin{definition}[\protect{\cite{GvT}}]\label{def:2.5}
A quasi-perfect $e$-error-correcting $q$-ary code $C$ is called
{\em uniformly packed} if there exist natural numbers $\lambda$ and $\mu$
such that for any vector $x$:
$$
B_{x,e+1} = \left\{
\begin{array}{cl}
  \lambda & \mbox{ if } d(x,C)=e \\
  \mu     & \mbox{ if } d(x,C)=e+1
\end{array}\right.
$$
\end{definition}

For the case $\mu = \lambda + 1$ any such binary code is
UP in the narrow sense.
Note that binary extended perfect codes fall into this category for
$\lambda=0$ and $\mu = (n+1)/(e+1)$. These
codes include also the ternary
Golay code, its extension, and the code obtained by shortenig.
Van Tilborg \cite{vTi} (see also \cite{Lind, SZZ})
showed that no other nontrivial codes of this kind exist for
$e>3$.

The following result is a generalization of the result in \cite{SZZ}.

\begin{proposition}[\protect{\cite{GvT}}]
A uniformly packed code is completely regular.
\end{proposition}

Note that the extension of any UP code, which is not perfect,
may not be uniformly packed. It was one of the motivation
of the following codes.

\begin{definition}[\protect{\cite{BZZ}}]
A code $C$ is {\em uniformly packed in the wide sense} if there exist
rational numbers $\beta_0,\ldots,\beta_\rho$ such that for any $x\in\F_q^n$
$$
\sum_{i=0}^\rho \beta_i B_{x,i}=1.
$$
\end{definition}

The above concept is much more general than the ones in \Cref{def:2.4,def:2.5}. In other words, a code $C$ is UP
in the wide sense if the all-one column is a linear combination of the
first $\rho$ columns of the outer distribution matrix $B$.
Later we will see that any CR code is UP in the wide sense.
Now, give some easy examples of CR codes.

\begin{enumerate}
\item As we have already mentioned, any perfect code is a
CR code.
\item The set of all even weight vectors is a linear CR
code with $\rho=e+1=1$.
\item An extended perfect code is a quasi-perfect UP code,
and so CR \cite{SZZ}.
\item For any CR code $C$, the subconstituent $C(\rho)$ is
also CR, with reversed intersection array \cite{Neum}.
\end{enumerate}

\subsection{Designs and distance-regular graphs}

Two related combinatorial structures are of special interest: $t$-designs and distance-regular graphs.

\begin{definition}
A $t$-$(v,k,\lambda)$-{\em design} is an incidence structure $(S,\B)$,
where $S$ is a $v$-set of elements (called {\em points}) and $\B$ is
a collection of $k$-subsets of points (called {\em blocks})
such that every $t$-subset
of points is contained in exactly $\lambda>0$ blocks ($0< t\leq
k\leq v$).
\end{definition}

In terms of incident matrix a $t$-$(v,k,\lambda)$-design is a binary code $C$
of length $n=v$ with codewords of weight $w=k$ such that any binary vector
of length $n$ and weight $t$ is covered by exactly $\lambda$ codewords.
A $t$-design with $\lambda=1$ is called a Steiner system and also
denoted by $S(v,k,t)$.
The following properties are well known and can be found for example in \cite{Beth, Blake, Hugh}.

\begin{proposition}
Given a $t$-$(v,k,\lambda)$-design, every $i$-subset of points ($0\leq
i\leq t$) is contained in exactly $\lambda_i$ blocks, where
$$
\lambda_i=\lambda\frac{\binom{v-i}{t-i}}{\binom{k-i}{t-i}}.
$$
\end{proposition}

\begin{corollary}
Given a $t$-$(v,k,\lambda)$-design ${\cal D}$: (i) ${\cal D}$ is an
$i$-design, for all $i\leq t$; (ii) $\lambda=\lambda_t$; (iii) the
number of blocks of ${\cal D}$ is: $b=\lambda_0$; (iv) each point
is contained in the same number of blocks, namely,
$r:=\lambda_1=bk/v$ ($r$ is called the {\em replication number}).
\end{corollary}

The first nontrivial $6$-designs were found in 1983 by Magliveras and
Leavitt \cite{MagL}. In 1987, L. Teirlinck proved that there are nontrivial
$t$-designs for any natural number $t$ \cite{Teir}.

%\begin{theorem}[Fisher's inequality]
%If $t>1$ and $2\leq k< v$, then the number of blocks is at least the number of points, i.e. $b\geq v$ and ($k\leq r$).
%\end{theorem}
%
%Ray-Chaudhuri and Wilson \cite{Rayc} proved that if $t=2s$ and $v\geq k+s$, then
%$b\geq \binom{v}{s}$.

There is a natural $q$-ary generalization of such $t$-designs (see \cite{AssGM,Dels,GvT,Rif1}).
Let $E=\{0,1, \ldots, q-1\}$. A collection $\B$ of $b$
vectors $x_1, \ldots,x_b$ of length $v$ and weight $k$ over $E$ is
called a $q$-ary $t$-design and denoted $t$-$(v,k,\lambda)_q$, if for every
vector $y$ over $E$ of length $v$ and weight $t$ there are exactly $\lambda$
vectors $x_{i_1}, \ldots, x_{i_\lambda}$ from $\B$ such that
$d(y,x_{i_j}) = k-t$ for all $j=1,\ldots,\lambda$. If $\lambda=1$,
then we obtain a $q$-ary Steiner system, denoted $S(v,k,t)_q$.

For a code $C$ denote by $C_w$ the set of all codewords of $C$ of weight $w$.
Regularity of a code $C$ implies that the sets $C_w$ induce $t$-designs.

\begin{theorem}[\protect{\cite{GvT}}]\label{disseny}
Let $C$ be a $t$-regular code with minimum distance $d\geq 2t$, then the supports
of any nonempty set $C_w$ form the blocks of a $t$-design.
\end{theorem}

Directly from definition of CR codes we have the following

\begin{theorem}\label{designs}
Let $C$ be a $q$-ary CR code of length $n$ with distance $d$.

\begin{itemize}
\item[(i)] If $d=2e+1$ then any nonempty set $C_w$ is an
$e$-$(n,w,\lambda_w)_q$-design.
\item[(ii)] If $d=2e+2$ then any nonempty
set $C_w$ is an $(e+1)$-$(n,w,\lambda_w)_q$-design.
\item[(iii)] If $C$
is a $q$-ary perfect code, then any nonempty set $C_w$ is an
$(e+1)$-$(n,w,\lambda_w)_q$-design and $C_d$ is a Steiner
system $S(n,2e+1,e+1)_q$.
\item[(iv)] If $C$ is an extended $q$-ary perfect
code, then any nonempty set $C_w$ is an
$(e+2)$-$(n,w,\lambda_w)_q$-design and $C_d$ is a Steiner
system $S(n,2e+2,e+2)_q$.
\end{itemize}
\end{theorem}
\bigskip

\bigskip

Let $\Gamma$ be a finite connected simple graph (i.e., undirected, without loops and multiple
edges). Let $d(\gamma, \delta)$ be the distance between two vertices
$\gamma$ and $\delta$ (i.e., the number of edges in the minimal path between
$\gamma$ and $\delta$). The {\em diameter} $D$ of $\Gamma$ is its largest distance.
Two vertices $\gamma$ and $\delta$ from $\Gamma$ are
{\em neighbors} if $d(\gamma, \delta) = 1$. Denote
\begin{eqnarray*}
\Gamma_i(\gamma) ~=~\{\delta \in \Gamma:~d(\gamma, \delta) = i\}.
\end{eqnarray*}
An {\em automorphism} of a graph $\Gamma$ is a
permutation $\pi$ of the vertex set of $\Gamma$ such that, for all
$\gamma, \delta \in \Gamma$ we have $d(\gamma,\delta)=1$ if and
only if $d(\pi\gamma,\pi\delta)=1$.

\begin{definition}[\protect{\cite{BCN}}]\label{14:de:1.3}
A simple connected graph $\Gamma$ is called {\em distance-regular}
if it is regular of valency $k,$ and if for any two vertices
$\gamma, \delta \in \Gamma$ at distance $i$ apart, there are precisely
$c_i$ neighbors of $\delta$ in $\Gamma_{i-1}(\gamma)$ and $b_i$
neighbors of $\delta$ in $\Gamma_{i+1}(\gamma)$.
Furthermore, this graph is called {\em distance-transitive}, if
for any pair of vertices $\gamma, \delta$ at distance
$d(\gamma, \delta)$ there is an automorphism $\pi$ from
$\mbox{\rm Aut}(\Gamma)$ which moves this pair $(\gamma,\delta)$ to any other given
pair $\gamma', \delta'$  of vertices at the same distance
$d(\gamma, \delta) = d(\gamma', \delta')$.
\end{definition}

The sequence $\{b_0, b_1, \ldots, b_{D-1};
c_1, c_2, \ldots, c_D\},$ where $D$ is the diameter of $\Gamma,$
is called the {\em intersection array} of $\Gamma$. The numbers
$c_i, b_i,$ and $a_i,$ where $a_i=k- b_i - c_i,$ are called
{\em intersection numbers}. Clearly $b_0 = k,~~b_D = c_0 = 0,~~c_1 = 1.$

Let $C$ be a linear CR code with covering radius $\rho$ and
intersection array  $\{b_0, \ldots , b_{\rho-1}; c_1, \ldots c_{\rho}\}$.
Let $\{A\}$ be the set of cosets of $C$. Define the graph $\Gamma_C,$ which is
called the {\em coset graph of $C$}, taking all different cosets $A = C+ x$ as
vertices, with two vertices $\gamma = \gamma(A)$ and $\gamma' = \gamma(A')$
adjacent if and only if the cosets $A$ and $A'$ contain neighbor vectors,
i.e., there are $v \in A$ and $ v' \in A'$ such that $d( v,  v') = 1$.

\subsection{Parameters and properties of CR codes}\label{props}

For a code $C$, we denote by $s+1$ the number of nonzero terms in the dual
distance distribution of $C$, obtained by MacWilliams' transform. The parameter
$s$ was called {\em external distance} by Delsarte \cite{Dels}, and is equal
to the number of nonzero weights of $C^\perp$ if $C$ is linear. This is a
key parameter as we will see in several properties. Recall that $B$ is the
outer distribution matrix of $C$ and $b+1$ denotes the number of different rows of $B$.

\begin{theorem}\label{params}
The following statements hold:
\begin{itemize}
\item[(i)] {\cite{Dels}} $\Rank(B)= s+1$.
\item[(ii)] {\cite{Sole}} $b\geq s$.
\item[(iii)] {\cite{Dels}} $\rho\leq s$.
\end{itemize}
\end{theorem}

Hence, we have the inequalities $e\leq\rho\leq s \leq b$. Now, we have the following characterization.

\begin{theorem}\label{equivs}
$\mbox{ }$
\begin{itemize}
\item[(i)] {\cite{Dels}} A code $C$ is perfect if and only if $e=s$.
\item[(ii)] {\cite{GvT}} A code $C$ is UP  if and only if $s=e+1$.
\item[(iii)] {\cite{Sole}} If $C$ is CR, then $\rho=s$.
\item[(iv)] {\cite{BZ}} $C$ is UP in the wide sense if and only if $\rho=s$.
\end{itemize}
\end{theorem}

The converse of (iii) is not true. Delsarte gives the example of a
$[48,24,12]$ extended quadratic residue code with $\rho=s=8$
and $b=14$ \cite{Dels}.  However the same condition is necessary and
sufficient for UP codes in the wide sense. Therefore, we obtain

\begin{corollary}
If $C$ is CR, then $C$ is UP in the wide sense.
\end{corollary}

Many UP in the wide sense codes which are not CR
are constructed in \cite{Rif2,Lift}.
The following properties are due to Delsarte \cite{Dels}.

\begin{theorem}[\protect{\cite{Dels}}]
$\mbox{ }$
\begin{itemize}
\item[(i)] If $t\geq d-s \geq 0$, then $C$ is $t$-regular.
\item[(ii)] If $C$ is $t$-regular, with $t\geq s-1$, and
$d\geq 2 s -1$, then $C$ is CR.
\end{itemize}
\end{theorem}

We can strengthen these conditions if all weights of $C$ are even.

\begin{corollary}
Let $C$ be an even code. Then,
\begin{itemize}
\item[(i)] \cite{BZ} If $t\geq d-s + 1 \geq 0$, then $C$ is $t$-regular.
\item[(ii)] \cite{BCN} If $d\geq 2s -2$, then $C$ is CR.
\end{itemize}
\end{corollary}

We say that a binary code $C$ of length $n$ is {\em self-complementary}
(also called {\em antipodal}), if for any codeword $c \in C$, there is
a codeword $\bar{c}$, which is at distance $n$ from $c$, i.e.
$\bar{c} = c + (1, \ldots, 1)$. Evidently, if $C$ is a binary CR code
and contains the codeword of weight $n$, then $C$ is self-complementary.
If $(1, \ldots, 1) \not \in C$, then $C$ is non-self-complementary.

\begin{theorem}\label{non-self-com}
Let $C$ be a non-self-complementary CR code with covering radius $\rho$.
\begin{itemize}
\item[(i)] \cite{Bor3} The set $C(\rho)$ is a translate of $C$ by $(1, \ldots, 1)$.
\item[(ii)] \cite{comb1} The set $C \cup C(\rho)$ is a CR code.
\end{itemize}
\end{theorem}

\subsection{Necessary conditions for CR codes}\label{Lloyd}

For a given CR code $C$ with covering radius $\rho$ and
intersection numbers $a_i,b_i,c_i$, a tridiagonal matrix $A$ which is called
the {\em intersection matrix} is defined as follows:
\[
A=\left[
\begin{array}{cccccc}
a_0    & b_0   &  0    & \ldots  &  0       & 0\\
c_1    & a_1   &  b_1  & \ldots  &  0       & 0 \\
 0     & c_2   &  a_2  & \ldots  &  0       & 0 \\
       &\ldots &\ldots & \ldots  &\ldots    & \\
 0     &  0    &  0    & \ldots  &a_{\rho-1}&b_{\rho-1}\\
 0     &  0    &  0    & \ldots  & c_\rho   & a_\rho\\
\end{array}
\right]\,.
\]
The following statement is called the Lloyd's theorem for
CR codes. Recall that the eigenvalues of the
Hamming cube $H_q(n)$ are the eigenvalues of the intersection
matrix of $H_q(n)$ which are equal to
$(q-1)n-q\,j$,\;$j=0,1,\ldots,n$.

\begin{theorem}[\protect{\cite{BCN,Neum}}]\label{Lloyd1}
Let $C$ be a CR code of length $n$ with intersection matrix $A$.
Then $A$ has $\rho$ integer eigenvalues, which are eigenvalues of $H_q(n)$.
\end{theorem}

Since any CR code is UP in the wide sense, there
exists another variant of this theorem which in some cases might be more useful
and which is a natural generalization of the classical Lloyd theorem for
perfect codes (see \cite{MacW}). Denote by $K_i$ the cardinality of $C(i)$.
Let $K_i= \kappa_i\,|C|$. It is easy to see that
\[
\kappa_i = \beta_i (q-1)^i\binom{n}{i},
\]
where $\beta_0, \beta_1, \ldots, \beta_\rho$ are parameters of a UP code
(see Definition 2.7).

\begin {theorem}[\protect{\cite{BZZ}}]\label{theo:1.1}
Let $C$ be a UP code in the wide sense of length $n$ with parameters
$\beta_0,\beta_1,\ldots\,\beta_\rho$. Then the polynomial
in $\xi$ of degree $\rho$
\EQ
L_\rho(n,\xi)\, =\, \sum_{r=0}^\rho ~\beta_r P_r(n,\xi),
\EN
where $P_r(n,\xi)$ is the Krawtchouk polynomial,
\[
P_r(n,\xi)\,=\,\sum_{j=o}^r
(-1)^{r-j}(q-1)^j\binom{n-\xi}{j}\binom{\xi}{r-j}\,,
\]
and for any real number $a$
\[
\binom{a}{i}\,=\,\frac{1}{i!} \; a\,(a-1)\ldots(a-i+1)\,,
\]
has $\rho$ distinct integer-valued  roots between $0$ and $n$.
\end{theorem}

The next theorem generalizes the classical sphere packing condition for perfect codes to UP codes
in the wide sense (hence, to any CR code).

\begin{theorem}[\protect{\cite{BZZ}}]\label{packing} Let $C$ be a uniformly packed code
in the wide sense of length $n$ with parameters
$\beta_0,\beta_1,\ldots\,\beta_\rho$. Then
\EQ
|C|~=~\frac{q^n}{\sum_{i=0}^\rho\;\beta_i(q-1)^i \binom{n}{i}}.
\EN
\end{theorem}

Some more interesting properties of CR codes
(which are also necessary conditions) can be found in \cite{BCN,Neum}.

We consider two illustrative examples \cite{BZZ}. The Preparata
$(n = 2^{2m} - 1, M = 2^{n+1-4m},d =5)$ codes $P$,~$m = 2, 3,
\ldots,$ have the following packing parameters $\beta_i$ and the
roots $\xi_i$ of the polynomial $P_\rho(n,\xi)$:
\[
\beta_0 =\beta_1 = 1,\; \beta_2 = \beta_3 = \frac{3}{n},
\]
\[
\xi_1 = \frac{1}{2}(n + 1 - \sqrt{n+1}),\;\;
\xi_2 = \frac{(n + 1)}{2},\;\; \xi_3  =  \frac{1}{2}(n+1+\sqrt{n + 1}).
\]
The interesting fact is that these codes are not only CR in $H_2(n)$ but
also in the Hamming code which contains this code $P$ (see \cite{SZZ}).

The BCH $[n = 2^{2m+1} - 1, k = n-4m-2, d = 5]$ codes  $m = 2, 3 ,
\ldots$, have the following parameters $\beta_i$ and $\xi_i$:
\[
\beta_0 = \beta_1 = 1, \; \beta_2 = \beta_3 = \frac{6}{(n-1)},
\]
\[
\xi_ 1 = \frac{n + 1}{2}\,-\, \sqrt{\frac{n + 1}{2}}, \;\;\xi_2 \,=\, \frac{n+1}{2}, \;\;\xi_3 \, = \, \frac{n+1}{2} + \sqrt{\frac{n+1}{2}}.
\]

\subsection{Completely transitive codes}\label{CTC}

Completely transitive (CT) codes were first introduced by Sol\'{e} \cite{Sole} as a
subclass of binary linear CR codes. If $C$ is a binary linear code,
then consider the natural action of $\Aut(C)$ over $\F_2^n/C$: for any coset $C+x$
and any $\sigma\in \Aut(C)$, set $\sigma(C+x)=\sigma(C)+\sigma(x)=C+\sigma(x)$.

\begin{definition}[\protect{\cite{Sole}}]\label{CTSole}
A binary linear code $C$ is {\em completely transitive} if $\Aut(C)$ gives $\rho+1$ orbits over $\F_2^n/C$.
\end{definition}

Since two cosets in the same orbit have identical weight distributions, we obtain

\begin{proposition}[\protect{\cite{Sole}}]
If $C$ is CT, then $C$ is CR.
\end{proposition}

The following fact is strengthening of Theorem \ref{non-self-com}.

\begin{proposition}[\protect{\cite{comb1}}]
If $C$ is a non-self-complementary CT code, then $C \cup C(\rho)$ is CT too.
\end{proposition}

There is a strong relation between CR codes and distance-regular graphs.

\begin{proposition}\label{lem:2.5}
Let $C$ be a linear CR code with covering radius $\rho$ and
IA = $\{b_0, \ldots , b_{\rho-1}; c_1, \ldots c_{\rho}\}$
and let $\Gamma_C$ be the coset graph of $C$. Then
\begin{enumerate}
\item[(i)]\cite{BCN}  $\Gamma_C$ is
distance-regular of diameter $D=\rho$ with the same IA.
\item[(ii)]\cite{ripu} If $C$ is CT, then $\Gamma_C$ is distance-transitive.
\end{enumerate}
\end{proposition}

Therefore, by \Cref{lem:2.5} any CR or CT code induces a
distance regular or distance transitive graph, respectively. Throughout this paper we will not deal with the associated distance regular graphs or distance transitive graphs to CR or CT codes, respectively. Those graphs, with the same intersection array than the respective codes can be see in corresponding references or in \cite{BCN,Dam2,koo}.

\bigskip

For a given permutation group $G$ of degree $n$ (acting on an $n$-set), we say
that $G$ is $t$-transitive (resp. $t$-homogeneous) if it sends any $t$-tuple (resp.
$t$-set) to any $t$-tuple (resp. $t$-set). $G$ is transitive if it is $1$-transitive.
A result of Livingston and Wagner \cite{LW} states that if $G$ is $i$-homogeneous,
with $i\leq n/2$, then $G$ is also $j$-homogeneous, for $j\leq i$. This fact implies

\begin{proposition}[\protect{\cite{Sole}}]\label{CTrho}
Let $C$ be a binary linear code of length $n$ and covering radius $\rho\leq n/2$.
If $\Aut(C)$ is $\rho$-homogeneous then $C$ is CT.
\end{proposition}

Using this property we obtain that the following codes are CT:

\begin{itemize}
\item[(i)] All perfect binary linear codes (repetition codes, Hamming codes and the binary Golay code).
\item[(ii)] All extended binary linear perfect codes.
\end{itemize}

\Cref{CTrho} gives us a sufficient (but not necessary) condition
for completely transitivity. It can be seen from the binary $[9,5,3]$
code $C$ which is dual to the code obtained by
the Kronecker product of
two $[3,2,2]$ parity check codes. Code $C$ is UP \cite{Sole,SZZ}
with $\rho=s=2$ and $e=1$. Moreover, $C$ is CT,
however $\Aut(C)$ is transitive but not $2$-homogeneous.

A necessary condition is the following one.

\begin{proposition}[\protect{\cite{Sole}}]\label{CTHomogeneous}
If $C$ is CT, then $\Aut(C)$ is $e$-ho\-mo\-ge\-ne\-ous.
\end{proposition}

%Let $C$ be the binary primitive cyclic $[31,21,5]$ code with generating polynomial $g(x)=x^{10}+x^9+x^8+x^6+x^5+x^3+1$. Such code has covering radius $\rho=3$ and it is completely regular \cite{Can1,Car2} with intersection array $(31,30,17;1,2,15)$ (this is the code for $\ell=4$ in the family of Subsection \ref{ABFunctions}). As can be seen in \cite{Ber0}, $\Aut(C)$ is the semilinear group of $\F_{2^5}$ over $\F_{2^5}$. Such permutation group is usually denoted $\Gamma L(1,32)$ and its order is $5|GL(1,32)|=

As an example of CR code which is not CT, consider a binary primitive
cyclic $[2^m-1,2^m-2m-1,5;3]$ code $C$ with $m>4$ and $m$
odd. Such code is a double-error-correcting BCH code which is CR \cite{BZ,SZZ}. The intersection array is
$\{n,n-1,(n+3)/2;1,2,(n-1)/2\}$, where $n=2^m-1$ (these are the codes for
$\ell=3$ in the family of Subsection \ref{ABFunctions}).
As can be seen
in \cite{Ber0}, $\Aut(C)$ is the semilinear group $\SL(1,2^m)$ of $\F_{2^m}$ over
$\F_{2^m}$ (remind that the semilinear group denoted $\SL(t,q)$ consists of all invertible semilinear transformations of $\F_q^t$ over $\F_q$).
If $q=p^r$, then it is well known that $|\SL(t,q)|=r|\GL(t,q)|$, where
$\GL(t,q)$ denotes the general linear group (see, for example, \cite[p. 163]{Beth}).
Hence, the order of $\Aut(C)$ is $|\SL(1,2^m)|=m|\GL(1,2^m)=m(2^m-1)$.
Since $C$
has packing radius $e=2$, we know that $C$ has exactly $(2^m-1)(2^{m-1}-1)$
cosets of minimum weight $2$. Therefore, as the number of such cosets is
greater than $|\Aut(C)|$, it is not possible that they are in the same
orbit by $\Aut(C)$. We conclude that $C$ is not CT.

Similarly to perfect and quasi-perfect UP codes, the nonexistence
of CT codes for $e>3$ was also established.
In 2000, Borges and Rif\`{a} \cite{Trans1} proved:

\begin{theorem}[\protect{\cite{Trans1}}]
If $C$ is a nontrivial CT code, then $e\leq 4$.
\end{theorem}

The proof was based on the nonexistence of highly transitive groups
and some bounds on the size of a code. Using the Griesmer bound and
the nonexistence of certain designs, the result was improved in 2001
by Borges, Rif\`{a} and Zinoviev \cite{Trans2}:

\begin{theorem}[\protect{\cite{Trans2}}]
If $C$ is a nontrivial CT code, then $e\leq 3$.
\end{theorem}

\bigskip

Clearly, \Cref{CTSole} can be extended to nonbinary linear codes. Giudici and Praeger \cite{Giu1,Giu2} studied this more general case. They called {\em coset-completely transitive} these codes (including Sol\'{e}'s binary case). In a previous preprint, Godsil and Praeger generalized the concept of complete transitivity to the nonlinear case. A newer version of this preprint is \cite{GodP}.

\begin{definition}[\protect{\cite{GodP}}]\label{GenCT}
A code $C\subset\F_q^n$ is $G$-{\em completely transitive}, or simply, {\em completely transitive} if there exists a subgroup $G$ of $\Aut(\F_q^n)$ such that each subconstituent (cell) $C(i)$ of the distance partition is a $G$-orbit.
\end{definition}

\begin{proposition}[\protect{\cite{Giu2}}]
If $C$ is completely transitive then it is CR.
\end{proposition}

Note that, in general, a coset-completely transitive code is not $\Aut(C)$-completely transitive in the sense of \Cref{GenCT}. This is because $\Aut(C)$ is often not even transitive on $C$, for example, when $C$ has codewords of different weights.

For a linear code $C\subset\F_q^n$, define $N_C$ as the set of all translations of $\F_q^n$
by vectors in $C$, i.e., $N_C=\{\tau_x\mid x\in C\}$, where $\tau_x(v)=v+x$, for every $v\in\F_q^n$. Clearly, $N_C$ is a subgroup, since $C$ is linear.
Define now the semidirect product $G=N_C \rtimes \Aut(C)$. $G$ fixes $C$ set-wise and the $N_C$-orbits are the cosets of $C$, in particular, $C$ is a $G$-orbit \cite{Giu2}. From all these observations, the following result is obtained:

\begin{theorem}[\protect{\cite{Giu2}}]
Let $C\subset\F_q^n$ be a linear code. Then $C$ is coset-completely transitive if and only if $C$ is $\left(N_C\rtimes \Aut(C)\right)$-completely transitive.
\end{theorem}

However, for $q\leq 3$, the concepts are equivalent:

\begin{theorem}[\protect{\cite{Giu2}}]
Let $C\subset\F_q^n$ be a linear code, where $q\leq 3$. Then $C$ is
coset-completely transitive if and only if $C$ is completely transitive.
\end{theorem}

Let $q\geq 7$ be a prime power, $q\neq 8$, and let $C$ be the repetition
code in $\F_q^3$. Taking $G=S_q\rtimes S_3$, it can be verified that $C$ is
$G$-completely transitive, however $C$ is not coset-completely transitive \cite{Giu2}.

Completely transitivity is a quite special property, which has no relation
to the optimality of codes. For example, the best after perfect, Preparata codes
(which have maximal possible packing number \cite{SZZ}) are not completely
transitive (except when they are the Nordstrom-Robinson code).

\begin{theorem}[\protect{\cite{GP2}}]
Let $C$ and $C^*$ be the Preparata code of length $n$ and its extension, respectively.
These codes are completely transitive if and only if $n=15$.
\end{theorem}

\section{Extension of CR codes}\label{ECR}

Given a binary code $C$, we define the extended code $C^*$ by adding a parity
(or antiparity) check bit to each codeword.
One of the interesting open questions of CR codes is the following one:
given a CR $(n,N,d)$ code $C$ with odd distance $d=2e+1$ under which conditions
its extension, i.e. the code $C^*$ is again a CR code.
Here we restrict our attention to binary codes,
although it seems that many results can be extended to the
nonbinary case.

In \cite{Bro1} it is proven that puncturing an even CR
code gives also a CR code (answering a question
posed in \cite{BCN}). Therefore, if $C^*$ is CR,
so is $C$. However, the converse is not true, in general.
Bassalygo and Zinoviev \cite{BZ} gave an example of a CR
code $C$ such that $C^*$ is not CR: the
double punctured binary Golay code is CR but its
extension is not. Moreover, in this case, the extended code is not
UP in the wide sense. Therefore the extension of a
UP code in the wide sense could be non-UP.
A necessary and sufficient condition is given in
\cite{BZ}.

\begin{theorem}[\protect{\cite{BZ}}]
A binary UP in the wide sense code $C$ of
length $n$ with packing parameters $\beta_0,\ldots,\beta_\rho$
remains to be UP under extension, if and only if the following
system of equations holds:
\[
\beta_{\rho-2i}=\beta_{\rho-2i-1}\;\;\mbox{for}\;\; i,\;\, 0\leq
i\leq[(\rho-1)/2].
\]
Furthermore, the packing parameters
$\gamma_0,\ldots,\gamma_\rho,\gamma_{\rho+1}$ of the extended code
$C^*$ are defined by the following formulas:
\[
\gamma_{\rho-2i}=\beta_{\rho-2i},\;\;\forall\; i=0,1,\ldots,\lfloor\rho/2\rfloor,
\]
\[
\gamma_{\rho-2i+1}=\frac{1}{n+1}\big((\rho+1-2i)\beta_{\rho-2i}+(n-\rho+2i)\,
\beta_{\rho-2i+2}\big),\;\forall\; i=0,\ldots,\big\lfloor\frac{\rho+1}{2}\big\rfloor.
\]
\end{theorem}

We give also a useful necessary condition for these codes. As we know
(Theorem \ref{designs}) the set $C_d$ of binary UP code $C$ of length $n$ and distance $d=2e+1$
induces an $e$-$(n,d,\lambda)$ design.

\begin{proposition}[\protect{\cite{BZ}}]
Let $C^*$ be the extension of a binary UP code in the wide sense
of length $n$ and minimum distance $d=2e+1$.  If
$C^*$ is UP, then the set $C^*_{d+1}$ induces
an $(e+1)$-$(n+1,d+1,\lambda)$-design.
\end{proposition}

In the case when $C$ and $C^*$ are UP and $\rho=e+1$, then they are also CR:

\begin{proposition}[\protect{\cite{GP1,SZZ}}]\label{exten-up}
Let $C$ be a quasi-perfect UP code (so, CR). If $C^*$ is UP, then
$C^*$ is CR.
\end{proposition}

Another necessary condition is the following.

\begin{proposition}
If $C$ is a CR code of even length and $C$ is self-comple\-men\-ta\-ry, then $C^*$ is
not UP in the wide sense (and hence not CR).
\end{proposition}

\begin{demo}
Assume that $C^*$ is UP, then the external distance is $s^*=s+1$ implying
that for each weight $w$ in the dual distribution, $n+1-w$ is also a weight in the
dual distribution. But $w$ must be even (since $C$ contains the all-one vector),
hence $n+1-w$ is odd, getting a contradiction.
\end{demo}

\bigskip

The following property is a strengthening of a result in \cite{BRZ14a}.

\begin{proposition}[\protect{\cite{BRZ15}}]
Let $C$ be a binary linear CR code of length $n=2^m-1$, minimum
distance $d=3$, covering radius $\rho=3$ and intersection array $\{n,b_1,1;1,c_2,n\}$.
Let the dual code $C^\perp$ have nonzero weights $w_1$, $w_2$ and $w_3$. Then
the extended code $C^*$ is CR with covering radius $\rho^*=4$ if
and only if $w_1 + w_3=2w_2=n+1$. In such case, the intersection array of $C^*$
is $\{n+1,n,b_1,1;1,c_2,n,n+1\}$.
\end{proposition}

Hence, the weights of the dual code (when $C$ is linear) play an
important role. Another important factor is that puncturing
the extended code at any coordinate should give almost the same code.

\begin{proposition}[\protect{\cite{Sole}}]
Let $C$ be a binary CR code. Assume that puncturing $C^*$ at any coordinate gives the same code $C$.
Let $s^*$ be the external distance of $C^*$. If $s^*\leq s+1$, then $C^*$ is CR.
\end{proposition}

\begin{corollary}[\protect{\cite{Sole}}]\label{ExtLin}
If $C$ is a binary linear CR code, the weights of $C^\perp$ are
even and symmetrical with respect to $(n+1)/2$, and if $\Aut(C^*)$ is transitive,
then $C^*$ is CR.
\end{corollary}

Let $C$ be a double-error-correcting $BCH$ code with parameters
$[2^m-1,2^m-1-2m,5]$, $m\geq 3$ odd. The weights of the dual code satisfy
the hypothesis of \Cref{ExtLin} and $C^*$ is left invariant by the
affine group. Hence $C^*$ is CR with $\rho^*=s^*=4$ and $e^*=2$.
This result can be deduced from \cite{BZZ}, where it was shown that these BCH codes
are UP in narrow sense with $\beta_2=\beta_3 = 6/(n-1)$ and, hence,
by \cite{SZZ} the extended codes are CR.

Let $C$ be the dual code of a three-weight cyclic code of length $n=2^m-1$ studied by
Calderbank and Goethals in \cite{CG1,CG2}. $C$ is CR and $C^*$ is
left invariant by the affine group. Again, the three weights of $C^\perp$ satisfy
the condition of \Cref{ExtLin}. Hence $C^*$ is CR with
$\rho^*=s^*=4$ and $e^*=1$.

\Cref{ExtLin} can be generalized to the nonlinear case:

\begin{corollary}[\protect{\cite{Sole}}]\label{ExtNLin}
Let $C$ be a binary code whose dual distances are even and symmetrical with
respect to $(n+1)/2$. If $C$ is CR and $\Aut(C^*)$ is transitive,
then $C^*$ is CR.
\end{corollary}

From the presentation of \cite{BvLW}, it is known that the automorphism group of
an extended Preparata code is transitive. The punctured code, i.e., the Preparata
code is CR \cite{SZZ} and its dual distances are those of the
Kerdock code satisfying the condition of \Cref{ExtNLin}. The conclusion
is that the extended Preparata code is CR \cite{SZZ} with
$\rho^*=s^*=4$ and $e^*=2$.

Let $C^*$ be a Hadamard $(12, 24, 6)$ code. It is known that $C$ is
UP with dual distances $4, 6, 8$ \cite{GvT, CCD}. $\Aut(C^*)$ is
isomorphic to the Mathieu group $M_{12}$. Hence $C^*$ is CR
with $\rho^*=s^*=4$ and $e^*=2$.

\begin{proposition}[\protect{\cite{GP1}}]
Let $C$ be a binary CR code with parameters $(n,d)$.
\begin{itemize}
\item[(i)] If $(n,d)=(12,6)$ then $C$ is equivalent to the Hadamard code.
\item[(ii)] If $(n,d)=(11,5)$ then $C$ is equivalent to the (punctured) Hadamard code.
\end{itemize}
Moreover, both codes are CT.
\end{proposition}

The results of \cite{Sole} were slightly improved in \cite{Bro2}.

\begin{proposition}[\protect{\cite{Bro2}}]
Let $C$ be a binary CR code.
If the outer distribution matrices of all codes obtained from $C^*$ by deleting
one coordinate position have the same set of rows (and, in particular, if $C^*$
admits a group transitive on the set of coordinate positions), then
$C^*$ is CR.
\end{proposition}

\begin{corollary}[\protect{\cite{Bro2}}]\label{EstBro}
$C^*$ is CR if and only if all codes obtained from it by deleting
one coordinate position are CR with the same outer distribution.
\end{corollary}

It is interesting to find necessary and sufficient conditions on $C$ for
$C^*$ to be CR. From the last corollary we can derive some
necessary conditions on the punctured codes of $C^*$, in particular on $C$.

For any binary vector $v=(v_1,\ldots,v_n)$ and each $i=1,\ldots,n$, define
$\tau_i(v)=(v_1,\ldots,v_{i-1},p(v),v_{i+1},\ldots,v_n)$, where $p(v)$
denotes the parity of $v$, i.e., $p(v)=\sum_{i=1}^n v_i$ (modulo 2). Define
the codes $C_{[i]}=\{\tau_i(x)\mid x\in C\}$.

\begin{lemma}
$C^*$ is a CR code if and only if $C$ and $C_{[i]}$ are
CR, for $i=1,\ldots,n$.
\end{lemma}

\begin{demo}
For any code $D$, denote by $D^{(i)}$ the punctured code obtained by
deleting the $i$th coordinate. Also, denote by $\sigma_{i,j}$ the
transposition of the coordinates $i$ and $j$. Assume that the parity
check coordinate is at position $n+1$. Hence, it is clear that
$C_{[i]}=\left(\sigma_{i,n+1}(C^*)\right)^{(i)}$ and $C=(C^*)^{(n+1)}$.
The result then follows by \Cref{EstBro}.
\end{demo}

\begin{proposition}
If $C^*$ is a CR code, then for all $i=1,\ldots,n$
\begin{itemize}
\item[(i)] The weight distributions of $C$ and $C_{[i]}$ coincide.
\item[(ii)] The minimum distances of $C$ and $C_{[i]}$ coincide and are odd.
\item[(iii)] The external distances of $C$ and $C_{[i]}$ coincide.
\item[(iv)] The covering radii of $C$ and $C_{[i]}$ coincide.
\end{itemize}
\end{proposition}

\begin{demo}
(i) follows since for each codeword $x$, $B_x$ must be the same for all codes $C$ and $C_{[i]}$.

(ii) and (iii) are direct from (i).

(iv) If the covering radius of a code $C$ is $\rho$, then the covering radius of $C^*$ is $\rho^*=\rho+1$ \cite{Bro2}. Hence the result follows.
\end{demo}

\bigskip

Now, we have the following necessary condition on $C$ (or $C_{[i]}$).

\begin{corollary}
If $C^*$ is CR code with minimum distance $d^*=2e+2\geq 4$, of length $n+1$, then
for all odd $w$
\begin{equation}\label{recursio}
(n-w)A_w = (w+1)A_{w+1},
\end{equation}
where $A_w$ is the number of codewords of weight $w\geq 2e+1$ in $C$ (or $C_{[i]}$).
\end{corollary}

\begin{demo}
Denote by $A^*_{w+1}$ the number of codewords in $C^*$ of weight $w+1$, $w$ odd.
This set $C^*_{w+1}$ of codewords of weight $w+1$ form a $2$-$(n+1,w+1,\lambda^*_2)$-design,
by \Cref{disseny}. The number of codewords in $C^*_{w+1}$ with nonzero value at position
$n+1$ is $r^*$, the replication number, and clearly $r^*=A_w$. Therefore,
\begin{equation}\label{replication}
A^*_{w+1}(w+1)=(n+1)r^*.
\end{equation}
Combining (\ref{replication}) with $A^*_{w+1}=A_{w+1}+A_{w}$, the result follows.
\end{demo}

\bigskip

In particular, any perfect code must satisfy (\ref{recursio}). For the case
of binary perfect codes with $d=3$ (or 1-perfect), this recursion is well
known (see, for example, \cite{MacW}):
\begin{equation}\label{recperfectes}
(n-i+1)A_{i-1}+A_i+(i+1)A_{i+1}=\binom{n}{i}.
\end{equation}

By combining (\ref{recursio}) and (\ref{recperfectes}), we obtain

\begin{corollary}
For any binary 1-perfect code containing the zero codeword, the number of codewords of weight $i$ is:
$$
A_i= \left\{
\begin{array}{cl}
  \binom{n}{i} \frac{1}{n-i+1} - A_{i-1} & \mbox{ if } i \mbox{ is odd;} \\
  & \\
  \frac{n-i+1}{i} A_{i-1} & \mbox{ if } i \mbox{ is even.}
\end{array}\right.
$$
\end{corollary}

By \cite{Bor3}, we know that the even half of a Golay code, say $C^*$, is
CR. Puncturing at any coordinate gives a code $C$ with
weight distribution:
$$
\begin{array}{|c|c|c|c|c|c|c|}
  A_0 & A_7 & A_8 & A_{11} & A_{12} & A_{15} & A_{16} \\
  \hline
  1 & 176 & 330 & 672 & 616 & 176 & 77
\end{array}
$$
Of course, $C$ is CR \cite{Bor3} and hence it verifies (\ref{recursio})
as can be readily seen.

\section{Constructions and existence of CR codes}\label{Cons}

In this section we give infinite families (numbered by (F.i)) and
sporadic cases (numbered by (S.i)) of known to us CR codes. For all codes
we give the intersection arrays (IA).

\subsection{CR codes from perfect codes}

Recall that for a code $C$ of length $n$, the punctured code at coordinate $i$
is obtained by deleting the coordinate $i$ in all the codewords. When the coordinate
$i$ is not specified, it is assumed that the resulting code is equivalent puncturing
at any coordinate. The shortened code of $C$ is obtained by taking all the codewords
that have a zero in a fixed coordinate and then, removing such coordinate. More
generally, for vectors $x_1,\ldots,x_r$ of length $j<n$, the
$\{x_1,\ldots,x_r\}$-shortened code of $C$ is obtained by taking all the codewords
that have $x_\ell$ (for any $\ell=1,\ldots, r$) in some fixed $j$ coordinates and
then, removing such coordinates.

\bigskip

From \cite{SZZ,Dels} we have the following well-known families of CR codes:

\begin{enumerate}[label=(F.\arabic*),series=family]
\item Any $q$-ary perfect $(n,q^n/(1+n(q-1)),3;1)_q$ code is a CR code with
$$
\IA = \{(q-1)n; 1\},\;\mbox{where}\;
q \geq 2, n \geq 3.
$$
\item Any $q$-ary extended perfect $(n+1,q^n/(1+n(q-1)),4;2)_q$ code is a CR code with
$$
\IA = \{(q-1)(n+1), (q-1)n; 1, 4\},\;\mbox{where}\;
q \geq 2, n \geq 4.
$$
\item Any $q$-ary $(n-1,q^n/(1+n(q-1)),2;1)$ code obtained by puncturing
any perfect $(n,q^n/(1+n(q-1)),3;1)$ code is a CR code with
$$
\IA = \{(q-1)(n-1); q)\},\;\mbox{where}\;
q, n \geq 3.
$$
\end{enumerate}

Let $D$ be any $q$-ary perfect $(n,N,3)_q$-code and let $C$ be any
$(n,N/q,4)$ subcode  of $D$ with the following property: for any
choice of the zero codeword in $D$, the set $C_4$ is a $q$-ary
$2$-$(n,4,\lambda)_q$ design where $\lambda = (n-3)/2$. Then, from \cite{BRZ15, Rif1}, we have:

\begin{enumerate}[label=(F.\arabic*),resume*=family]
\item\label{halfs} $C$ is a CR code with
$$
\IA = \{n(q-1),(n-1)(q-1),1; 1,(n-1),n(q-1)\},
$$
where $q \geq 2,\;n \geq 4$.
\item In particular, for any prime power $q=p^s \geq 3$ where $s=1,2,\ldots$,
there exists a CR $[q +1,q-2,4;3]_q$ code, a subcode of a perfect $[q+1, q-1,3]_q$-code, with
$$
\IA = \{q^2-1, q(q-1), 1; 1, q, q^2-1\}.
$$
\item \label{shortcode} Any $q$-ary $(n,N,3;2)_q$ code, obtained by shortening
a perfect $(n+1, qN, 3;1)_q$ code $C$, is a CR code with
$$
\IA = \{(q-1)n, q-1; 1, (q-1)n\}.
$$
\end{enumerate}

Now, we see several different halves of a
binary perfect code giving CR codes with different intersection arrays.

Let $C$ be a binary perfect $(n,N,3)$ code.
From \cite{Bor3}, we have that the even or odd half of $C$ is a CR code (included in the family \ref{halfs}). Also from \cite{Bor3}, we have that the punctured code of the even half of $C$ is a CR code (included in family \ref{shortcode}).

Other halves of binary Hamming codes can be obtained with the following procedure. Let $H_m$
denote the parity check matrix of a binary Hamming code
of length $n=2^m-1$. For a given even $m \geq 4$ and any  $i_1,
i_2 \in \{0,1,2,3\}$, where $i_1 \neq i_2$, denote by
$\bv_{i_1,i_2} = (v_0,v_i,\ldots, v_{n-1})$ the binary vector whose
$i$-th position $v_i$ is a function of the value of the weight of
the column $\bh_i$:
\[
v_i~=~\left\{
\begin{array}{ccc}
1, ~~&\mbox{if}&~~\wt(\bh_i) \equiv i_1,~\mbox{or}~ i_2 \pmod{4},\\
0, ~~&  &  ~~\mbox{otherwise}.
\end{array}
\right.
\]

Let $C$ be the binary perfect linear $[n=2^m-1,k=n-m,3]$ code with $m$ even and parity check matrix $H_m(\bv_{i_1,i_2})$, obtained from $H_m$ by adding one more row
$\bv_{i_1,i_2}$. Then, from this code we have \cite{BRZ14a}:
\begin{enumerate}[label=(F.\arabic*),resume*=family]
\item\label{antipodalhalf} If $\{i_1,i_2\}=\{1,2\}$ or $\{2,3\}$, then $C$ is a self-complementary half of a Hamming code, which is a CR $[n, n-m-1,3;3]$-code with
\[
\IA = \{n, (n+1)/2,1;1, (n+1)/2, n\}.
\]
\item If $\{i_1,i_2\}=\{0,1\}$ or $\{0,3\}$, then $C$ is a non-self-complementary half of a Hamming code, which is a CR $[n, n-m-1,3;3]$-code with
\[
\IA = \{n, (n-3)/2,1;1, (n-3)/2, n\}.
\]
\end{enumerate}

As can be seen in \cite{BRZ14a}, if $\{i_1,i_2\}=\{0,2\}$, then $C$ is the even half of a Hamming code, hence included in Family \ref{halfs}. If $\{i_1,i_2\}=\{1,3\}$, then $C$ is the Hamming code $[n,n-m,3]$.

Extensions of  the codes \ref{antipodalhalf} are also CR \cite{BRZ14a}:

\begin{enumerate}[label=(F.\arabic*),resume*=family]
\item The extension of the code \ref{antipodalhalf} is a
self-complementary CR $[n+1,n-m,4;4]$ code with
\[
\IA = \{n+1, n, (n+1)/2,1;1, (n+1)/2, n, n+1\}.
\]
\end{enumerate}

From \cite{Zve, Rif1}, we have several families obtained by shortening binary perfect or extended perfect codes.

Let $C^*$ be any binary extended perfect $(n^*,N^*,4)$ code, $n^*=2^m \geq 8$.
\begin{enumerate}[label=(F.\arabic*),resume*=family]
\item Let $C$ be the $(n = n^*-2,N=N^*/2,2;3)$ code obtained by
$\{(00),(11)\}$-shortening $C^*$. Then $C$ is a CR code with
$$
\IA = \{n, n-2, 2; 2, n-2, n\},\;\mbox{where}\;n=2^m-2 \geq 6.
$$
\item Let $C$ be the $(n = n^*-3, N^*/4, 1;2)$ code obtained by $\{(000),(111)\}$-shor\-te\-ning
$C^*$. Then $C$ is a CR code with
$$
\IA = \{n-1, 3; 1, n-1\},\;\mbox{ where }\;n=2^m-3 \geq 5.
$$
\item Let $D$ be any binary perfect $(n,N,3)$ code, $n=2^m-1 \geq 7$.
Let $C$ be the $(n -2,N/2,1;2)$ code obtained by
$\{(00),(11)\}$-shortening $D$. Then $C$ is a CR code with
$$\IA = \{n-3, 2; 2, n-3\}.
$$
\end{enumerate}

Also from \cite{Zve, Rif1}, we have a family obtained by shortening $q$-ary extended perfect codes:

\begin{enumerate}[label=(F.\arabic*),resume*=family]
\item Let $C^*$ be any $q$-ary extended perfect $[n^*,k^*,4;2]_q$ code where
$q=2^m \geq 4$, $n^*=q+2$ and $k^*=q-1$. Let $C$ be the
$[n = q, k=q-2,2;2]_q$ code obtained by $S$-shortening $C^*$, where
$S =\{(\alpha,\alpha):\;\alpha \in \F_q\}$. Then
$C$ is a CR code with
$$\IA = \{q(q-1), (q-1)(q-2); 2, q\}\;\mbox{where}\;q=2^m \geq 4.$$
\end{enumerate}

Now we give sporadic
CR codes, which come
from Golay codes \cite{Rif1}. The complete regularity of the codes \ref{halfGolay} and \ref{puncturedhalfGolay} were stated in \cite{Bor3}.

\begin{enumerate}[label=(S.\arabic*),series=sporadic]
\item {\em The binary Golay code.} This perfect $[23,12,7;3]$
code is CR with
$$ \IA = \{23, 22, 21; 1, 2, 3\}.$$
\item {\em The binary punctured Golay code.} This $[22, 12, 6; 3]$
code is CR with
$$ \IA = \{22, 21, 20; 1, 2, 6\}.$$
\item {\em The binary extended Golay code.} This $[24, 12, 8; 4]$
code is CR with
$$ \IA = \{24, 23, 22, 21; 1, 2, 3, 24\}.$$
\item {\em The binary double punctured Golay code.} This $[21, 12, 5; 3]$
code is CR with
$$ \IA = \{21, 20, 16; 1, 2, 12\}.$$

\item\label{halfGolay} {\em The half of the binary Golay code.} This $[23,11,8;7]$ code is
CR with
$$
\IA = \{23, 22, 21, 20, 3, 2, 1; 1, 2, 3, 20, 21, 22, 23\}.
$$
\item\label{puncturedhalfGolay} {\em The punctured of the half of the binary  Golay code.} This
$[22, 11, 7; 6]$ code is CR with
$$
\IA = \{22, 21, 20, 3, 2, 1; 1, 2, 3, 20, 21, 22\}.
$$
\item
{\em The $\{(00,11)\}$-shortened binary extended Golay
code.} This $[22, 11, 6; 7]$ code is CR with
$$
\IA = \{22, 21, 20, 16, 6, 2, 1; 1, 2, 6, 16, 20, 21, 22\}.
$$
\item
{\em The $\{(000,111)\}$-shortened binary extended Golay
code.} This is a CR $[21, 10, 5; 6]$ code with
$$
\IA = \{21, 20, 16, 9, 2, 1; 1, 2, 3, 16, 20, 21\}.
$$
\item
{\em The $\{(00,11)\}$-shortened binary Golay code.} This
$[21, 11, 5; 6]$ code is CR with
$$
\IA = \{21, 20, 16, 6, 2, 1; 1, 2, 6, 16, 20, 21\}.
$$
\end{enumerate}

Let $G$ denote the ternary perfect Golay $[11,6,5]_3$-code and
denote by $G^{(0)}$ the subcode of $G$, formed by all codewords of
$G$ with parity $0$. It is easy to see that $G^{(0)}$ is the
$[11,5,6]_3$ code, formed by all codewords of weights $0, 6$ and
$9$. Call this code the third part of the ternary Golay.

\begin{enumerate}[label=(S.\arabic*),resume*=sporadic]

\item
{\em The ternary Golay code.} This perfect $[11,6,5; 2]_3$
code is CR with
$$\IA = \{22, 20; 1, 2\}.$$
\item
{\em The ternary punctured Golay code.} This $[10,6,4; 2]_3$
code is CR with
$$\IA = (20, 18; 1, 6).$$
\item
{\em The ternary extended Golay code.} This $[12,6,6; 3]_3$
code is CR with
$$\IA = \{24, 22, 20; 1, 2, 12\}.$$
\item
 {\em The third part of the ternary Golay code.} The ternary
$[11, 5, 6; 5]_3$-code $G^{(0)}$ is CR with
$$ \IA = \{22, 20, 18, 2, 1; 1, 2, 9, 20, 22\}.$$
\item
{\em The punctured code of the third part of the ternary Golay.} The
ternary $[10, 5, 5; 4]_3$ code is CR with
$$ \IA = \{20, 18, 4, 1; 1, 2, 18, 20\}$$.
\end{enumerate}

There are many codes considered in later sections, based in some way on perfect codes (e.g. \Cref{sec:lifting,sec:nested,sec:concat,sec:rho1,sec:rho2}).

\subsection{Nested families}\label{sec:nested}

Recall that $\mathcal{H}_m$ is a binary Hamming code of length
$n=2^m-1$. Let $m=2u$,  $q=2^u,$
$r=2^u+1$ and $\rr=2^u-1$. We can think of the parity check matrix
$H_m$ of $\mathcal{H}_m$ as the binary representation of
$[\alpha^0, \alpha^1, \ldots, \alpha^{n-1}],$ where $\alpha \in
\F_{2^m}$ is a primitive element. Present the elements of
$\F_{2^m}$ as elements in a quadratic extension of $\F_{2^u}$. Let
$\beta=\alpha^r$  be a primitive element of $\F_{2^u}$ and let
$\F_{2^m}=\F_{2^u}[\alpha]$.
Let $E_m$ be the binary representation of the matrix
$[\alpha^{0r},\alpha^{r},\ldots,\alpha^{(n-1)r}]$. Take the matrix
$P_m$ as the vertical join of $H_m$ and $E_m$.

It is well known~\cite{CG1} that the code $C^{(u)}$ with parity
check matrix $P_m$ is a cyclic binary CR code with
covering radius $\rho=3,$ minimum distance $d=3$ and dimension
$n-(m+u)$.

It can be seen \cite{BRZ15} that the number of cosets $C^{(u)}+
v,$ of weight three, is $\rr$. Indeed, their syndromes $S(v)$ are
the nonzero elements of $\F_{2^u}$. For $i\in \{0,\ldots,u\},$
taking $u-i$ cosets $C^{(u)}+v^{(1)}, \ldots, C^{(u)}+v^{(u-i)}$
with independent syndromes $S(v^{(1)}), \ldots, S(v^{(u-i)})$
(independent, means that they are independent binary vectors in
$\F_2^u$) we can generate a linear binary code $C^{(i)}=\langle
C^{(u)}, v^{(1)}, \ldots v^{(u-i)}\rangle$. Let $A_{u-i}$ be the
linear subspace of $\F_{2}^u$ generated by the syndromes
$S(v^{(1)}), \ldots, S(v^{(u-i)})$.

The dimension of code $C^{(i)}$ is $\dim(C^{(i)})=
u-i+\dim(C^{(u)}),$ where $\dim(C^{(u)})=n-m-u$. Note that the
maximum number of independent syndromes we can take is $u,$ so the
biggest code we can obtain is of dimension $u+\dim(C^{(u)})=n-m,$
which is the Hamming code $C^{(0)}=\mathcal{H}_m$. All the
constructed codes contain  $C^{(u)}$ and, at the same time, they
are contained in the Hamming code $C^{(0)}$.

The number of codes $C^{(u-i)}$ equals the number of subspaces of
dimension $i$ we can take in $\F_2^u,$ so the Gaussian binomial
coefficient

$$
|\{C^{(u-i)}\}|=\binom{u}{i}_{\!\!2}=\frac{(2^u-1)(2^u-2)\cdots(2^u-2^{i-1})}{(2^i-1)(2^i-2)\cdots(2^i-2^{i-1})}.
$$
The number of different nested families
of codes between $C^{(u)}$ and $C^{(0)}=\mathcal{H}_m$, we can construct,
equals
$$
\prod_{i=0}^{u-1} (2^{u-i}-1)\,.
$$

The following property was stated in \cite{CG1} for the code
$C^{(u)},$ but it can be extended to all codes $C^{(i)},$ for
$i\in \{1,\ldots, u\}$.

\begin{proposition}\label{calder}
For $i\in \{1,\ldots, u\}$ the cosets of weight three of $C^{(i)}$
are at distance three from each other and $C^{(i)}\cup C^{(i)}(3)$
is the Hamming code.
\end{proposition}

\begin{theorem} [\protect{\cite{BRZ15}}]
$\mbox{ }$
\begin{enumerate}
\item[(i)]  $C^{(1)}$ is a CT code with covering radius 3.
\item[(ii)] $C^{(u)}$ is a CT
code with covering radius 3.
\item[(iii)]   For
$i\in \{0,\ldots,u\},$ the code $C^{(i)*}$ is CT when $C^{(i)}$ is CT.
\item[(iv)]   For
$i\in \{1,\ldots,u\},$ the code $C^{(i)}$ is a subcode of $C^{(i-1)}$,
and $C^{(i)*}$ is a subcode of $C^{(i-1)*}$.
\end{enumerate}
\begin{enumerate}[label=(F.\arabic*),resume*=family]
\item   For $i\in
\{0,\ldots,u\},$ the code $C^{(i)}$ is CR with
$$ \IA = \{2^m-1,2^{m}-2^{m-i},1;1,2^{m-i},2^m-1\}.$$
\item   For $i\in \{0,\ldots,u\},$ the extended code  $C^{(i)*}$
is CR with
$$ \IA = \{2^m,2^m-1,2^{m}-2^{m-i},1;1,2^{m-i},2^m-1,2^m\}.$$
\end{enumerate}

\end{theorem}

Note that,  for $i\in\{0,1,u\}$ the codes $C^{(i)}$ and $C^{(i)*}$
are CT. Also, for $m=6,$ all codes $C^{(i)}$ and $C^{(i)*}$ are CT.
In \cite{iva} for the graphs (to be distance transitive) coming from
such codes (to be CT ) were obtained the following
divisibility conditions:  $2^m$ is a power of $2^i$, or $2^m=2^i$, or $2^i-1$ divides $2m$.
Therefore, we conjecture that when one of such divisibility conditions is
satisfied, then  $C^{(i)*}$ is a CT code. Moreover,
in such cases, we conjecture that $C^{(i)}$ is also CT.
However, the question about complete transitivity of codes
$C^{(i)}$ and $C^{(i)*}$ for $i\neq 0,1,u$ is open and needs more
attention.

\subsection{ Preparata-like and BCH codes}\label{subsec:preparata}

From \cite{SZZ} and \cite{BZZ} we have

\begin{enumerate}[label=(F.\arabic*),resume*=family]

\item Any (binary) Preparata-like $(n=2^{2m}-1,N=2^{n+1-4m},5;3)$ code ($m \geq 2$)
is CR with
\[
\IA = \{n,n-1,1;1,2,3\}.
\]
\item An extended Preparata-like $(n+1=2^{2m},N=2^{n+1-4m},6;4)$ code is CR with
\[
\IA = \{n+1,n,n-1,1;1,2,3,n+1\}.
\]
\item Primitive binary BCH $(n=2^{2m+1}-1,N=2^{n-4m},5;3)$ codes ($m \geq 2$)
are CR with
\[
\IA = \{n,n-1,(n+3)/2;1,2,(n-1)/2\}.
\]
\item Extended primitive BCH $(n+1=2^{2m+1},N=2^{n-4m},6;4)$ codes are
CR with
\[
\IA = \{n+1,n,n-1,(n+3)/2;1,2,(n-1)/2,n+1\}.
\]

\end{enumerate}

\subsection{Lifting Hamming codes}\label{sec:lifting}

CR codes can be obtained by lifting Hamming codes.
Denote by $H^q_m$ the parity check matrix of the Hamming code
$C=C(H^q_m)$ of length $n=(q^m-1)/(q-1)$ over $\F_q$. Define a new
linear code, denoted $C_r(H^q_m)$, of length $n$ over $\F_{q^r}$,~$r \geq 2$,
with this parity check matrix $H^q_m$.

\begin{theorem}[\protect{\cite{Lift}}]\label{theo:2.2}
$\mbox{ }$
\begin{enumerate}[label=(F.\arabic*),resume*=family]
\item The code $C_r(H^q_m)$ is a CR
$[n,n-m,3;\rho]_{q^r}$ code with $\rho = \min\{r,m\}$ and intersection numbers:
$$
b_i=\frac{(q^r-q^{i})(q^m-q^{i})}{(q-1)},\;i=0, \ldots,\rho-1; \; c_i=q^{i-1}\frac{q^i-1}{q-1},\;i=1, \ldots,\rho.
$$
When $r\neq m$, codes $C_r(H^q_m)$ and $C_m(H^q_r)$ are not
equivalent, but they have the same intersection array.
\end{enumerate}
\end{theorem}

Note that Hamming codes are the only codes whose lifting give
CR codes.
\begin{theorem}[\protect{\cite{Lift}}]\label{theo:2.3}
Let $C(H^q)$ be the nontrivial code of length $n$ over the field $\F_q$
with minimum distance $d \geq 3$, with covering radius $\rho \geq
1$ and $C_r(H^q)$ is its lifting over
$\F_{q^r}$. Then $C_r(H^q)$ is CR, if and only if $C(H^q)$ is a Hamming code.
\end{theorem}

Using \Cref{theo:2.3} the codes obtained by lifting extended
perfect codes never give CR codes. However,
these codes are UP in the wide sense
\cite{Lift}.

The next statement generalizes results of
\cite{BZ, BZZ} to the nonbinary case.

\begin{proposition}[\protect{\cite{Lift}}]\label{prop:3.3}
Let $C$ be the $q$-ary Hamming $[n, n-m, 3]$ code,
$n=(q^m-1)/(q-1)$ and $C^*$ be its extended code. The
code $C^*$ is UP if and only if the minimum distance
is $4$. In other words, the code $C^*$ is UP if and only
if $q=2$ with $m \geq 2$ or $q=2^u \geq 4$ with $m=2$.
\end{proposition}

\begin{theorem}[\protect{\cite{Lift}}]
Let $C$ be the Hamming $[n, n - m, 3]_q$-code of length $n =
(q^m-1)/(q-1)$ and let $C^*$ be the extended code. The lifted
code $C^*_r$ is a UP code if and only if $C^*$
is it. Hence, the lifted code $C^*_r$ is a UP
code if and only if $q=2$ with $m \geq 2$ or $q = 2^u \geq 4$ with
$m=2$.
\end{theorem}

\subsection{Kronecker product construction}

In \cite{Rif2} a Kronecker construction of CR
codes has been investigated and, later, in \cite{kro2} the
construction has been extended taking different alphabets in the
component codes. This approach is also connected with lifting
constructions of CR codes.
One interesting thing is that several classes of
CR codes with different parameters, but identical
intersection array, are obtained.

\begin{theorem}[\protect{\cite{Rif2,kro2}}]\label{theo:4.1}
Let $C(H^{q^u}_{m_a})$ and $C(H^q_{m_b})$ be two Hamming codes
with parameters $[n_a,n_a-m_a,3]_{q^u}$ and $[n_b,n_b-m_b,3]_q$,
respectively, where $n_a = (q^{u\,m_a}-1)/(q^u-1)$,~
$n_b=(q^{m_b}-1)/(q-1)$,~ $q$ is a prime power,~ $m_a, m_b \geq
2$, and $u \geq 1$.

\begin{enumerate}[label=(F.\arabic*),resume*=family]
\item The code $C=C(H)$ with parity check matrix
$H = H^{q^u}_{m_a} \otimes H^q_{m_b}$, the Kronecker product of
$H^{q^u}_{m_a}$ and $H^q_{m_b}$, is  CT, and
so CR, $[n, k, d;\rho]_{q^u}$ code with parameters
\begin{equation}\label{eq:4.11}
n=n_a\,n_b,~~k=n-m_a\,m_b,~~d=3,~~\rho=\min\{u\,m_a,~m_b\}.
\end{equation}
and with intersection numbers:
\[
b_\ell ~=~
\frac{(q^{u\,m_a}-q^{\ell})(q^{m_b}-q^{\ell})}{(q-1)},\;\;\;\ell =
0, 1, \ldots, \rho-1,
\]
and
\[
c_{\ell} ~=~ q^{\ell-1}\frac{q^\ell-1}{q-1},\;\;\;\ell = 1, 2,
\ldots, \rho.
\]
The lifted code $C_{m_b}(H^{q}_{u m_a})$ is  CR
with the same IA as $C$.
\end{enumerate}
\end{theorem}

Remark that in the above \Cref{theo:4.1}  we can not
choose the code $C_{m_b}(H^{q^u}_{m_a})$ (instead of
$C_{m_b}(H^{q}_{um_a})$), which seems to be natural. We emphasize
that the codes $C_{m_b}(H^{q}_{um_a})$ and
$C_{m_b}(H^{q^u}_{m_a})$ are not only different CR
codes, but they induce different distance-regular graphs with
different intersection arrays. So, the code
$C_{m_b}(H^{q}_{um_a})$ suits to the codes (F.22) in the sense
that it has the same intersection array. For example, the code
$C_2(H_3^{2^2})$ induces a distance-regular graph with
intersection array $\{315,240;1,20\}$ and the code $C_2(H_6^2)$
gives a distance-regular graph with intersection array
$\{189,124;1,6\}$. To obtain these results in both cases we use the
same \Cref{theo:2.2}.

The above theorem (\Cref{theo:4.1}) can not be extended to the
more general case when the alphabets $\F_{q^a}$ and $\F_{q^b}$ of
component codes $C_A$ and $C_B$, respectively, neither $\F_{q^a}$
is a subfield of $\F_{q^b}$ or vice versa $\F_{q^b}$ is a subfield
of $\F_{q^a}$. We illustrate it by considering the smallest
nontrivial example. Take two Hamming codes, the $[5,3,3]$ code
$C_A$ over $\F_{2^2}$ with parity check matrix $H_2^{2^2}$, and
the $[9,7,3]$ code $C_B$ over $\F_{2^3}$ with parity check matrix
$H_2^{2^3}$. Then the resulting $[45,41,3]$ code $C = C(H_2^{2^2}
\otimes H_2^{2^3})$ over $\F_{2^6}$ is not even UP
in the wide sense, since it has covering radius $\rho=3$ and
external distance $s=7$, which can be checked by considering the
parity check matrix of $C$.

\begin{theorem} [\protect{\cite{kro2}}]\label{theo:main}
Let $q$ be any prime number and let $a,b,u$ be any natural
numbers. Then there exist the following CR codes with
different parameters $[n,k,d;\rho]_{q'}$, where $q'$ is a power of
$q$, $d=3$, and
$\rho=\min\{ua,b\}$:

\begin{enumerate}[label=(F.\arabic*),resume*=family]
\item $C_{ua}(H_b^{q})\;\mbox{over
$\F_q^{ua}$ with}\;n = \frac{q^{b}-1}{q-1},\;k=n - b;$
\item
$C_b(H_{ua}^{q})\;\mbox{over $\F_q^{b}$ with}\;n =
\frac{q^{ua}-1}{q-1},\;k=n - ua;$
\item $C(H^q_b \otimes H^q_{ua}
)\;\mbox{over $\F_q$ with}\;n = \frac{q^b-1}{q-1} \cdot \frac {q^{ua}-1}{q-1}\,
 ,\;k=n - b ua;$
\item $C(H^q_b \otimes
H^{q^a}_{u})\;\mbox{over $\F_q^{a}$ with}\;n = \frac{q^b-1}{q-1}
\cdot \frac{q^{ua}-1}{q^a-1},\;k=n - bu;$
\item $C(H^q_b \otimes
H^{q^u}_a)\;\mbox{over $\F_q^{u}$ with}\;n = \frac{q^b-1}{q-1}
\cdot \frac{q^{ua}-1}{q^u-1},\;k=n - b a;$
\end{enumerate}
All the above codes have the same intersection numbers
\[
b_\ell=\frac{(q^b-q^{\ell})(q^{ua}-q^{\ell})}{(q-1)},\;\ell=0,
\ldots, \rho-1,\;\;
c_{\ell}=q^{\ell-1}\frac{q^\ell-1}{q-1},\;\ell=1, \ldots, \rho.
\]
All codes above coming from Kronecker constructions are CT.
\end{theorem}

 Denote by $\tau(n)$ the number
of divisors of $n$.

\begin{corollary}[\protect{\cite{kro2}}]
Given a prime power $q$ choose any two natural numbers $a,b >1$.
For each divisor $r$ of $a$ or $b$ we build the following $\tau(a)+\tau(b)$
different CR codes with identical intersection array and covering radius
$\rho=\min\{a,b\}$:
\begin{enumerate}[label=(\roman*)]
\item CT codes $C(H^{q^{r^*}}_r
\otimes H^q_b)$ over $\F_{q^{r^*}}$, for any proper divisors of $a$ with $rr^*=a$.
\item CT
codes $C(H^q_a \otimes H^{q^{r^*}}_r)$ over
$\F_{q^{r^*}}$, for any proper divisors of $b$ with $rr^*=b$.
\item CR codes $C_{a}(H^{q}_{b})$ over
$\F_{q^{a}}$ and $C_{b}(H^{q}_{a})$ over $\F_{q^{b}}$.
\end{enumerate}
\end{corollary}

This construction gives also UP in the wide sense codes which are
not CR.
\begin{theorem}[\protect{\cite{kro2}}]\label{theo:4.3}
Let $C(H_m^{q^u})$ be the $q^u$-ary Hamming $[n,k,3]_{q^u}$-code
of length $n_a=(q^{um}-1)/(q^u-1)$ and $C(R^q_{n_b})$ be the
repetition $[n_b,1,n_b]_q$-code, where $q$ is a prime power,~$u
\geq 1$,\, $m \geq 2$,\, $4 \leq n_b \leq (q^u-1)n_a + 1$.
\begin{itemize}
\item[(i)] The code $C = C(H_m^{q^u}\otimes R^q_{n_b})$ is a $q^u$-ary
UP (in the wide sense) $[n, k, d]_{q^u}$-code with
covering radius $\rho=n_b-1$ and parameters
\begin{equation}\label{eq:4.1}
n = n_a\,n_b,~~k = n - m\,(n_b-1),~~d = 3.
\end{equation}
\item[(ii)] The code $C$ is not CR.
\end{itemize}
\end{theorem}

\subsection{Binomial CR codes}

Denote by $H^{(m, \ell)}$ the binary matrix of size $m \times
\binom{m}{\ell}$, whose columns are all different vectors
of length $m$ and weight $\ell$. Define the binary linear code
$C^{(m, \ell)}$ with parity check matrix
$H^{(m,\ell)}$.

\begin{theorem}[\protect{\cite{comb}}]\label{theo}
Let $m$ and $\ell$ be two natural numbers such that $2\leq \ell \leq m-2$.
Code $C^{(m, \ell)}$ is CT (and CR) exactly in the following four cases:
\begin{enumerate}[label=(F.\arabic*),resume*=family]
\item For any $m \geq 4$, the code $C^{(m, 2)}$ is a $[n, k, d; \rho]$-code with
parameters:
$$
n~=~\binom{m}{2},~~k~=~n- m + 1,~~d=3,~~\rho=\lfloor m/2 \rfloor~.
$$
and with intersection numbers, for $i = 0, \ldots, \rho$:
\begin{eqnarray*}
a_{i}=2i(m-2i);\,\,
b_{i}=\binom{m-2i}{2};\,\,
c_{i}=\binom{2i}{2}\,.
\end{eqnarray*}
\end{enumerate}

\begin{enumerate}[label=(S.\arabic*),resume*=sporadic]

\item  The code $C^{(5,3)}$ is the $[10,5,4;3]$-code with
$$ \IA = \{10, 9, 4; 1, 6, 10\}$$.
\item  The code $C^{(6,4)}$ is the $[15,10,3;3]$-code with
$$ \IA = \{15, 8, 1; 1, 8, 15\}$$.
\item  The code $C^{(7,4)}$ is the $[35, 29, 3; 2]$-code with
$$ \IA = \{35,16;1,20\}$$.
\end{enumerate}
\end{theorem}

In fact, from the codes with $\ell=2$ we have some more CR
codes. First, we divide
these codes into  two families \cite{comb}.

\begin{theorem} [\protect{\cite{comb}}]\label{theo:new1}
Let $m$ be a natural number, $m \geq 3$. Let $C^{(m)} = C^{(m,2)}$.
The code $C^{(m)}$ is self-complementary if $m$ is odd and non-self-complementary if
$m$ is even.
\end{theorem}

Since for even $m$ the code $C^{(m)}$ is non-antipodal, its covering
set $C^{(m)}(\rho)$ is a translate of $C^{(m)}$ (Theorem \ref{non-self-com}).
Consider the new (linear) code
$C^{[m]} = C^{(m)} \cup C^{(m)}(\rho)$. The generating matrix
$G^{[m]}$ of this code has a very symmetric structure:
\[
G^{[m]}~=~\left[
\begin{array}{ccccc}
&I_{k-1}\,&|&\,H^t_{m-1}&\\\hline
&0 \ldots 0\,&|&\,1 \ldots 1&
\end{array}
\right]\,.
\]
Using  that $C^{(m)}(\rho) =
C^{(m)} + (1,1, \ldots, 1)$ (Theorem \ref{non-self-com}), we obtain

\begin{theorem} [\protect{\cite{comb1}}]\label{theo:new2}
Let $m$ be even, $m \geq 6$ and let $C^{[m]} = C^{(m)}
\cup C^{(m)}(\rho)$.
\begin{enumerate}[label=(F.\arabic*),resume*=family]
\item  Code $C^{[m]}$ is CT (and CR) $[n,k,d;\rho]$ code with parameters
\[
n~=~m(m-1)/2,~~k~=~n - m + 2,~~d~=~3,~~\rho~=~
\lfloor m/4 \rfloor.
\]
The intersection numbers of $C^{[m]}$ for
$m \equiv 0 \pmod{4}$ and $\rho = m/4$ are
\[
b_{i}~=~\binom{m-2i}{2},\;\; c_{i}~=~\binom{2i}{2},\;\; i = 0,1, \ldots, \rho-1,\;\;
c_{\rho}~=~2\,\binom{2\rho}{2},
\]
and for $m \equiv 2 \pmod{4}$ and $\rho = (m-2)/4$ are
\[
b_{i}~=~\binom{m-2i}{2},\;\; c_{i}~=~\binom{2i}{2},\;\;i = 0,1, \ldots, \rho\,.
\]
\end{enumerate}
\end{theorem}

\subsection{CR codes by direct sum construction}\label{sec:sum}

Let $C_1$ and $C_2$ be two codes, not necessarily linear, of the
same length $n$. The direct sum of $C_1$ and $C_2$ is the code
defined by:
$$
C_1 \oplus C_2=\{(c_1,c_2)\mid c_1\in C_1, c_2\in C_2\}.
$$

\begin{theorem}[\protect{\cite{BZZ,Sole}}]\label{directsum}
Let $u$ be any positive integer and let $C_i$, $i=1,2,...,u$ be
$q$-ary CR $(n,N,d;1)_q$ codes with the same intersection array $(b_0; c_1)$.
\begin{enumerate}[label=(F.\arabic*),resume*=family]
\item Then, for any $u \geq 1$, their direct
sum $C = C_1 \oplus \cdots \oplus C_u$ is a CR $(n u, N^u,d;u)_q$ code
with intersection numbers
for $i = 0,1,...,u$
\[
a(u)_i = n(q-1) - (u-i)b_0 - i c_1,~ b(u)_i = (u-i)b_0,~ c(u)_i =
i c_1.
\]
\end{enumerate}
\end{theorem}

We remark that the construction of \Cref{directsum} was used in \cite{Sole} for the particular case where the codes $C_i$ are binary perfect codes.

\subsection{ CR codes from combinatorial configurations}

\begin{enumerate}[label=(F.\arabic*),resume*=family]
\item  {\em One Latin square codes.} For any $q \geq 2$ a $q$-ary
MDS $(3,2,q^2;1)_q$ code is CR with
$$ \IA = \{3(q-1); 3\}.$$
In this case the set
$C(\rho)$ is the rest of $F^3$: $C(\rho) = F^3 \setminus C$.

\item \label{Latin2} {\em Two Latin square codes.} For any $q \geq 3$ and
$q \neq 6$ a $q$-ary MDS $(4,3,q^2;2)_q$ code is CR with
$$\IA = \{4(q-1), 3(q-3); 1, 12\}.$$
\end{enumerate}

\begin{enumerate}[label=(S.\arabic*),resume*=sporadic]
\item {\em Three Latin square code.} Three orthogonal Latin
squares of order $4$ form the equidistant $[5,4,16;3]_4$ code. This
code is CR with
$$ \IA = \{15, 12, 3; 1, 4, 15\}.$$

The punctured $[4,3,16]_4$ code, obtained from the code above
by deleting any one position is also a CR code
and it belongs to the family \ref{Latin2}.
\item {\em Four Latin squares code.} Four orthogonal Latin
squares of order $5$ form the equidistant $[6,5,25;3]_5$ code. This
code is CR with
$$ \IA = \{24, 20, 13; 1, 2, 6\}.$$

The $[5,4,25]_5$ code $C^{(p)}$ obtained by puncturing of this code
above is not CR. The subset $C_4$ is not a $2$-design.
But the $[4,3,25;2]_5$ code, obtained by double puncturing, is CR,
and belongs to the family \ref{Latin2}.

\item {\em The Hadamard code.} The unique Hadamard matrix of
order $12$ induces the CR binary $(11, 24, 5;3)$
code $H$ (see \cite{GvT}) with
$$ \IA = \{11,10,3; 1,2,9\}.$$

\item {\em The extended Hadamard code.} The $(12,24,6;4)$ code
$H^{*}$ obtained by extension of $H$ is also CR (see \cite{BZ}) with
$$ \IA = \{12, 11, 10, 3; 1, 2, 9, 12\}.$$

\end{enumerate}

\begin{enumerate}[label=(F.\arabic*),resume*=family]

\item {\em Constant weight codes.} For any natural $g$,
$g \geq 1$,  the trivial constant weight $(n,N,d;g)$ code
with
\[
n = 2g,\;N = \binom{2g}{g},\; d = 2,
\]
is CR with

$$\IA = \{2g, 2g-1,2g-2,\ldots,1; g+1,g+2,\ldots,2g\}.$$

Since the set of all
binary vectors of length $2g$ and weight $g$ is the Johnson scheme
$J(2g,g)$ this example shows that in this special case
the Johnson scheme $J(2g,g)$ is CR in the Hamming
scheme $H(2,2g)$.
\end{enumerate}

\subsection{ CR codes by concatenation constructions}\label{sec:concat}

In this subsection we collect some results from \cite{BRZ2017,BRZ2017a} dealing with CR constructed using concatenations of Hamming codes.

For any vector $x=(x_1,\ldots,x_n)\in \F_q^n$, denote by
$\sigma(x)$ the right cyclic shift of $x$, i.e.
$\sigma(x)=(x_n,x_1,\ldots,x_{n-1})$. Define recursively
$\sigma^i(x)=\sigma(\sigma^{i-1}(x))$, for $i=2,3,\ldots$ and
$\sigma^1(x)=\sigma(x)$. For $j<0$, we define
$\sigma^j(x)=\sigma^{\ell}(x)$, where $\ell = j$ modulo $n$.

Let $H$ be the parity check matrix of a $q$-ary cyclic Hamming
code of length $n=(q^k-1)/(q-1)$, (hence $\Gcd(n,q-1)=1$). Thus, the
simplex code generated by $H$ is also a cyclic code. Denote by
$\br_1,\ldots,\br_k$ the rows of $H$. For any $c\in
\{2,\ldots,n\}$, consider the code $C$ with parity check matrix
\begin{equation}\label{code:33}
\left[
\begin{array}{ccccc}
H\;& \;H \;&\ldots\,&\,H\,\\
H_1\;& \;H_2 \;&\ldots\,&\,H_c
\end{array}
\right],
\end{equation}
where $H_i$ is the matrix $H$ after cyclically shifting $i$ times its
columns to the right. In other words, the rows of $H_i$
are $\sigma^i(\br_1),\ldots,\sigma^i(\br_k)$.

\begin{enumerate}[label=(F.\arabic*),resume*=family]
\item \label{family:33}
The code $C$ with parity check matrix given in
(\ref{code:33}) is a CR code with parameters
$[nc,nc-2k,3;2]_q$ and intersection array
$$
IA=\{(q-1)nc,((q-1)n-c+2)(c-1);1,c(c-1)\}.
$$
\end{enumerate}

\begin{remark}
Almost all codes described in \ref{family:33} are not CT codes. However,
in the binary case and for any value of $k$ (so $n=2^k-1$), the CT
codes  are those with $c\in\{2,3,n-1,n\}$. In
general, in the $q$-ary case when $q$ is a power of two, the CT
codes are those with $c\in \{2,3\}$ and if $q=p^r$, for $p\not=2$,
then the CT codes are those with $c=2$.
\end{remark}

\begin{remark}
By extending the codes given in \ref{family:33} we
do not obtain CR codes, except for the binary case
when  $c$ equals $2^{k-1}+1$. In this case, the resulting
extended $[n(2^{k-1}+1)+1,n(2^{k-1}+1)-2k,4;3]$ codes coincide with the codes obtained in \ref{speccase}.
\end{remark}

Now, take the matrix
\begin{equation}\label{code:34}
H^{(k,c)} = \left[
\begin{array}{ccccc}
H\;& \;0 \;&\,H\,&\,H\,\ldots \,&\,H\\
0\;& \;H \;&\,H\,&\,H_1\,\ldots \,&\,H_c
\end{array}
\right]
\end{equation}
where $0$ denote the zero matrix (of the same size as $H$).

\begin{enumerate}[label=(F.\arabic*),resume*=family]
\item \label{family:34}
For $c \leq n-1$, the code $C^{(k,c)}$ with parity check matrix given in
(\ref{code:34}) is a CR code with parameters
 $[(c+3)n,(c+3)n-2k,3;2]_q$ and intersection array
$$ \IA = \{(c+3)(q-1)n, (c+2)((q-1)n-1-c); 1, (c+2)(c+3)\}.$$
In the binary case, when $c=n-1$ the code $C^{(k,n-1)}$ coincides with
the $[2^{2k}-1,2^{2k}-1-2k,3;1]$ Hamming code.
\end{enumerate}

 \begin{remark}
Almost all codes in the family  \ref{family:34} are not CT. However, in the binary case
and for any value of $k>2$, the codes $C^{(k,c)}$ are CT for $c\in \{2^k-5,2^k-4,2^k-3\}$.
\end{remark}

If we consider the extension of the codes $C^{(k,c)}$ given in \ref{family:34}
we obtain non CR codes in almost all cases. However, in the binary case and for each value
of $k$, there are exactly two values of $c$ such that the obtained extended code is CR.

Let $C$ be a binary Hamming $[n,k,3;1]$ code, where $n=2^k-1$ and let $H$ be its parity check
matrix. As in the family \ref{family:34}, take the parity check matrix $H^{(k,c)}$, where now $c\in \{2^{k-1}-2,2^{k}-2\}$.

\begin{enumerate}[label=(F.\arabic*),resume*=family]
\item \label{speccase}
Let $C^{(k,c)}$ be the code with parity check matrix $H^{(k,c)}$ and $(C^{(k,c)})^*$ its extended code.
 For $c=2^{k-1}-2$, the code $(C^{(k,c)})^*$ is a $[(c+3)n+1,(c+3)n-2k,4;3]$ CR code with
$$ \IA = \{(c+3)n+1, (c+3)n,2^{2k-2}; 1, (c+2)(c+3),(c+3)n+1\}.$$

For $c=2^{k}-2$, the code $(C^{(k,c)})^*$ is a $[(c+3)n+1,(c+3)n-2k,4;2]$ CR code which
coincides with the extended $[2^{2k},2^{2k}-2k,4;2]$ Hamming code.

For $c\notin \{2^{k-1}-2,2^{k}-2\}$ the code $(C^{(k,c)})^*$ is not CR.
\end{enumerate}

\begin{remark}
All extended codes $(C^{(k,c)})^*$ in the above family \ref{speccase} are not CT.
\end{remark}

We also know a few sporadic examples of CR codes constructed by using concatenation methods.

\begin{enumerate}[label=(S.\arabic*),resume*=sporadic]
\item \label{sporadic:s22} The binary $[15,9,3;3]$-code $C$ with parity check matrix
\begin{equation*}
H = \left[
\begin{array}{ccccc}
\;K\;&0\;&0\;&K\;&K\;\\
\;0\;&K\;&0\;&K\;&K_1\;\\
\;0\;&0\;&K\;&K\;&K_2\;
\end{array}
\right],\;\;\mbox{where}\;\;
K = \left[
\begin{array}{ccc}
\;1\;&0\;&1\;\\
\;0\;&1\;&1\;
\end{array}
\right]
\end{equation*}

and $K_1$ (respectively, $K_2$) is obtained by one cyclic shift of
the columns of $K$ (respectively, by two cyclic shifts) is
CR with
$$ \IA = \{15,12,1; 1,4,15\}.$$

\item The binary $[16,9,4;4]$-code, obtained by extension of the afore
mentioned $[15,9,3]$ code, is CR with
$$\IA = \{16,15,12,1; 1,4,15,16\}.$$
\end{enumerate}

Denote by $D(u,q)$ a difference matrix \cite{Beth}, i.e. a
square matrix of the order $qu$ over an additive group of order
$q$, such that the component-wise difference of any two rows
contains any element of the group exactly $u$ times.

Take the difference matrix
$D(2,3)$
\[
D=\left[
\begin{array}{cccccc}
0~&~0~&~0~&~0~&~0&~0\\
0~&0~&~1~&~2~&~2~&~1\\
0~&1~&~0~&~1~&~2~&~2\\
0~&2~&~1~&~0~&~1~&~2\\
0~&2~&~2~&~1~&~0~&~1\\
0~&1~&~2~&~2~&~1~&~0\\
\end{array}
\right]
\]

\begin{enumerate}[label=(S.\arabic*),resume*=sporadic]
\item \label{sporadic:S24}
Let $H$ be a binary $(12 \times 18)$ matrix obtained from $D$ by
changing any element $i$ by the matrix $K_i$ (which is the same as in \ref{sporadic:s22}). Then the
$[18,12,3;2]$ code with parity check matrix $H$, is a CR code
with
$\IA = \{18,15;1,6\}.$

\item Do the same construction as in \ref{sporadic:S24} for the matrix $D^*$, which is the difference matrix $D(2,3)$
without the trivial column.
The resulting $[15,9,3;3]$ code is CR with $$ \IA = \{15,12,1;1,4,15\}.$$ This code coincides with the code described in \ref{sporadic:s22}.
\end{enumerate}

\subsection{$q$-Ary CR linear codes with $\rho=1$}\label{sec:rho1}

Linear $q$-ary CR codes with $\rho=1$ are fully classified \cite{BRZ10}
by using the two simple following constructions.

{\em Construction $I(u)$.} Let $C$ be a $[n,k,d]_q$ code with a parity
check matrix $H$. Define a new code $C^{+u}$ with parameters
$[n+u,k+u,1]_q$ as the code with parity check matrix $H^{+u}$,
obtained by adding $u>0$ zero columns to $H$.

\begin{theorem}[\protect{\cite{BRZ10}}]
The codes $C$ and $C^{+u}$ have the same covering radius and,
moreover, $C$ is CR if and only if $C^{+u}$ is
CR. In this case, both codes have the same
intersection numbers $b_i=b_i'$ and $c_i=c_i'$, i.e.
\[
a'_i~=~a_i+(q-1)u,~~b'_i~=~b_i,~~c'_i~=~c_i,\;\;i = 0,1,\ldots,\rho
\]
(here $a_i, b_i, c_i$ are intersection numbers of $C$).
\end{theorem}

{\em Construction $II(\ell)$.}~ Let $C$ be a $[n,k,d]_q$ code with
parity check matrix $H$. Let
$C^{\times \ell}$ be the code with parameters
$[n\,\ell,k+(\ell-1)n,2]_q$, whose parity check matrix, denoted
$H^{\times \ell}$, is $\ell$ times the repetition of $H$ (or
monomially equivalent matrices of $H$), i.e.
$$
H^{\times \ell}~=~[H^{(1)}\,|\,H^{(2)}\,|\,\cdots \,|\,H^{(\ell)}],
$$
where, for all $i=1,\ldots,\ell$, $H^{(i)}$ is the parity check
matrix of an equivalent code to $C$.

\begin{theorem}[\protect{\cite{BRZ10}}]
An $[n,k,d]_q$ code $C$ is CR with covering radius
$\rho=1$ if and only if $C^{\times \ell}$ is CR with
covering radius $\rho'=1$.
\end{theorem}

Now we have the following classification theorem.

\begin{theorem}[\protect{\cite{BRZ10}}]\label{class-th-rho1}
Let $C=C(H)$ be a nontrivial $[n,k,d]_q$ code with covering radius
$\rho=1$ and with parity check matrix $H$.
\begin{enumerate}[label=(F.\arabic*),resume*=family]
\item The code $C$ is a $[n,n-m,d;1]_q$ CT (and CR) code, where $n=n_m\,\ell+u$ and $n_m=(q^m-1)/(q-1)$, if and only if the matrix $H$ is of the form
\[
H~=~\left((H^q_m)^{\times \ell}\right)^{+u}
\]
(up to monomial equivalence), where $H^q_m$ is a parity check matrix of
a Hamming code of length $n_m$ over $\F_q$.
\end{enumerate}
Furthermore,
\begin{enumerate}[label=(F.\arabic*),resume*=family]
\item $d=3$, if and only if $u=0$, $\ell = 1$, $n = n_{m}$
and $C$ is a Hamming code.
\item $d=2$, if and only if $u=0$, $\ell \geq 2$, $n =
n_{m}\ell$.
\item $d=1$, if and only if $u>0$, $\ell \geq 1$.
\end{enumerate}
In all cases the code $C$ has
\[
IA=\{(q-1)\ell\, n_{n-k};\ell\}
\]
\end{theorem}

Similar result was also obtained in \cite{koo2} in terms of arithmetic
CR codes.

\subsection{$q$-Ary CR nonlinear codes with $\rho=1$}

A coset of a linear CR code with $\rho=1$ obviously gives a
nonlinear such code with the same $\rho=1$.  Apart from these
trivial codes very little is known for this case. There are some
results for the binary case mostly due to Fon-Der-Flaas
\cite{FDF1,FDF2,FDF3}. The equivalent definition for the
construction of CR codes with given $\rho$ is the so called
perfect $(\rho+1)$-colorings of a hypercube. Especially simple
perfect colorings are defined for the case $\rho=1$, i.e., for
$2$-colorings. Let $H(2,n)$ be a binary hypercube of dimension
$n$. Its vertices are binary vectors of length $n$, and two
vertices are neighbors, if the corresponding vectors are at
distance $1$ from each other. Coloring of its vertices into white
and black colors is called a perfect $2$-coloring with
intersection matrix
\[
\left(
\begin{array}{cc}
a\,&\,b\\
c\,&\,d
\end{array}
\right)\,,
\]
if every black vertex has $a$ black and $b$ white neighbors, and
every white vertex has $c$ black and $d$ white neighbors. Clearly,
$a + b = c + d = n$. In our terminology, $a_0 = a$, $b_0= b$,
$c_1=c$, and $a_1 = d$ and, hence, it is a nonlinear CR code with
$\IA=\{b;c\}$. The following result gives the lower bound for the
value $a=n-b$ (this is the best known bound for {\em correlation
immunity}; see references in \cite{FDF1}).

\begin{theorem}[\protect{\cite{FDF1}}]
Let $C$  be a binary CR code of length $n$ and $\rho=1$ with $\IA=\{b; c\}$. If $b \neq c$, then
\EQ\label{correlimmun} c - a \leq \frac{n}{3}\,. \EN
\end{theorem}

The other necessary condition is the following result from \cite{FDF2}.

\begin{theorem}[\protect{\cite{FDF2}}]
Let $C$  be a binary CR code of length $n$ with $\rho=1$ and intersection
array $(b; c)$. Then $b \neq 0 \neq c$, and
\EQ\label{divisibility}
\frac{b + c}{(b,c)}\;\;\mbox{is a power of $2$}\,,
\EN
where $(b,c)$ is the greatest common divisor.
\end{theorem}

Both constructions for linear codes with $\rho=1$ considered in the
previous section work for nonlinear codes also.
The following statement generalizes the corresponding results from \cite{BRZ10}
for nonlinear case and from \cite{FDF2} for nonbinary case.

\begin{proposition}\label{prop:5.10}
$\mbox{ }$
\begin{enumerate}[label=(\emph{\roman*})]
\item For every $n=(q^m-1)/(q-1)$,\;$m\geq 2$, and any $k$,\;$1 \leq k \leq (q-1)n$,
there exists a $q$-ary CR code $C^{\cup k}$ with $\rho=1$ and
$\IA = \{(q-1)n-k+1; k \}$, formed by $k$ arbitrary translates of $q$-ary
perfect code of length $n$ and minimum distance $d=3$.
\item The existence of $q$-ary CR code $C$ of length $n$ with
$\IA=\{ b; c\}$ implies the existence of CR code $C^{+k}$ with
$\IA = \{b+k(q-1); c+k(q-1)\}$ for any $k\geq 1$, formed by changing of the every codeword
$c \in C$ by $q^k$ codewords of the form: $(c\,|\,x)$ where $x$ runs over $\F_q^n$.
\item \label{prop:4.21} The existence of a CR code $C$ of length $n$ with
$\IA = \{ b; c\}$ implies the existence, for any $k\geq
1$, of a CR code $C^{\times k}$ with $\IA = \{(k b; k c \}$. To every codeword
$v = (v_1, \ldots, v_n)$ of $C$, the all
vectors\\ $(v_{1,1},v_{1,2}, \ldots, v_{1,k}, \ldots,
v_{1,1},v_{1,2}, \ldots, v_{1,k})$
from $\F_q^{k\, n}$ are associated, for which
\[
\sum_{j=1}^k v_{i,j} = v_i.
\]
\end{enumerate}
\end{proposition}

The codes which meet the bound (\ref{correlimmun}) are the most
interesting. Such codes of length $n=3\,u$ should have
$\IA = \{ 3u-a; a+u \}$ where $a \geq 0$. Two infinite
families of codes with such intersection arrays are known (see
references in \cite{FDF1}): $\{3k;k\}$ and $\{5k;3k\}$.
These families come from codes with intersection arrays $\{3;1\}$
and $\{5; 3\}$ applying \Cref{prop:5.10}, \ref{prop:4.21}.
The first code is the trivial binary perfect code of length $3$, and the
second of length $6$ (constructed by Tarannikov \cite{tar}) can be
constructed from the first one using the two following lemmas due
to Fon-Der-Flaass \cite{FDF2}.

\begin{lemma}[\protect{\cite{FDF2}}]\label{flaass1}
Let $C$ be a CR code with $\rho=1$ and intersection matrix
\[
\left(
\begin{array}{cc}
a\,&\,b\\
c\,&\,d
\end{array}
\right)\,.
\]
Let $C$ can be partitioned into $k$-faces, $0 \leq k \leq a$. Then
there exists a CR code with $\rho=2$ and intersection matrix
\[
\left(
\begin{array}{ccc}
a-k\,&\,a+k\,&\,2b\\
a+k\,&\,a-k\,&\,2b\\
c\,&\,c\,&\,2d
\end{array}
\right)\,.
\]
\end{lemma}

\begin{lemma}[\protect{\cite{FDF2}}]\label{flaass2}
Let $C$ be a CR code with $\rho=2$ and intersection matrix
\[
\left(
\begin{array}{ccc}
a-k\,&\,a+k\,&\,2b\\
a+k\,&\,a-k\,&\,2b\\
c\,&\,c\,&\,2d
\end{array}
\right)\,,
\]
where $c \geq a+k$. Then there exists
a CR code with $\rho=1$ and intersection matrix
\[
\left(
\begin{array}{cc}
a-k\,&\,2b+c\\
c\,&\,2d+2c-a-k
\end{array}
\right)\,.
\]
\end{lemma}

In connection with condition (\ref{divisibility})
a natural question arises \cite{FDF2}: {\em to find the value $a^*(b,c)$ for all pairs $b,c$,
satisfying the condition (\ref{divisibility})}. Here under $a^*(b,c)$ we mean the minimum
value, such that there exists a code with $\IA = \{b; c \}$, if and only
if $a\geq a^*(b,c)$. There are lower and upper bounds for the value $a^*(b,c)$ (see \cite{FDF2}
and references there). The next statements gives the best known such bounds.

\begin{theorem}[\protect{\cite{FDF2}}]\label{lowerbounds}
~
\begin{itemize}
\item[(i)] If $c<b<2c$, then
\[
a^*(b,c) \geq  \frac{1}{2} + \sqrt{c(b-c) +\frac{1}{4}} - (b-c).
\]
\item[(ii)] If $2c<b<2c + \sqrt{3c-2}$, then $a^*(b,c) \geq 1$.
\end{itemize}
\end{theorem}

For a given nonzero integers $x,y$, such that $x+y=2^k-1$, and one
of them odd with $\ell$ consecutive ones in its binary presentation
(where $1\leq \ell <k$), define $z(x,y) = k-1$.

\begin{theorem}[\protect{\cite{FDF2}}]\label{upperbounds}
~
\begin{itemize}
\item[(i)] $a^*(2b+c,c) \leq \max (0, a^*(b,c) - 1)$.
\item[(ii)] Let $b$ and $c$ satisfy the condition (\ref{divisibility}). Then
$$a^*(c,b) = a^*(b,c) + b - c.$$
\item[(iii)] Let $b$ and $c$ satisfy the condition (\ref{divisibility}) and $b \geq c$, \,$m = (b,c)$,\,$b=m(2b'+1)$, $c=m(2c'+1)$.

If $c'=0$, then $a^*(b,c) = 0$. Otherwise
\[
a^*(b,c) \leq \max(0,c-m(z(b',c')+1)),
\]
\end{itemize}
\end{theorem}

An optimal binary CR code of length $n=12$ with $\IA = \{9; 7 \}$ meeting the
bound (\ref{correlimmun}) was constructed in \cite{FDF3}.
In the same paper, it was proven also that the putative CR code
of length $n=12$ with $\IA = \{11; 5 \}$
does not exist.

\subsection{$q$-Ary linear CR codes with $\rho=2$}\label{sec:rho2}

The dual code of any linear two-weight code could be a CR
code with $\rho=2$. There are many different families of such codes,
and their classification is not finished (see \cite{cald} for a survey
of such codes).

The classification of linear CR codes with
covering radius $\rho=1$, enabled the classification of linear CR codes
with covering radius $\rho=2$, whose dual codes are
self-complementary.

\begin{theorem}[\protect{\cite{BRZ10}}]
Let $C=C(H)$ be a nontrivial $[n,k,d]_q$ code.
 Then, $C$ is CR with covering radius
$\rho=2$ and the dual code $C^{\perp}$ is self-complementary if and only if
its parity check matrix $H$ looks, up to equivalence, as follows:
\[
H=\left[
\begin{array}{cccc}
\,1 \,&\,\cdots \,&\,1\,\\
\, &\,\,M\,\,&\,\,\\
\end{array}
\right],
\]
where $M$ generates an equidistant code $E$ with the following
property: for any nonzero codeword $v\in E$, every symbol
$\alpha\in\F_q$, which occurs in a coordinate position of $v$,
occurs in this codeword exactly $n-\tilde{d}$ times, where
$\tilde{d}$ is the minimum distance of $E$. Moreover, up to
equivalence, $C$ is the extension of a CR code $C'$
with covering radius $\rho'=1$.
\end{theorem}

\begin{corollary}[\protect{\cite{BRZ10}}]\label{enum}
The known CR codes with covering radius $\rho=2$ whose dual codes are
self-complementary are the following ones.
\begin{enumerate}[label=(F.\arabic*),resume*=family]
\item The binary extended perfect $[n,k,4;2]_2$
code ${\mathcal H}^*_{m}$ of length $n = 2^m$, where $k=n-m-1$ and
$m \geq 2$ with
$$
\IA = \{n,n-1;1,n\}.
$$

\item The extended perfect $[n,k,4;2]_q$
code ${\mathcal H}^*_{m}$ of length $n=q+2$ with $k=q-1$, where
$q=2^r \geq 4$, and $m=2$ \cite{bush,del1} (the family $TF1$ in \cite{cald}) with
$$
\IA = \{(q+2)(q-1),q^2-1;1,q+2\}.
$$

\item The $[n,k,3;2]_q$ dual of a difference matrix code with
$$
\IA = \{n(q-1),n-1;1,n(q-1)\}.
$$
It has length
$n = q^m$, dimension $k = n - (m+1)$ and parity check matrix $D_m$,
where $m \geq 1$, and $q \geq 3$ is any prime power (the code
generated by the matrix $D_m$ has been given in \cite{Sem1}). The
complementary code of this code is the Hamming code ${\mathcal
H}_{m}$.

\item The $[n,n-2,3;2]_q$ dual of a latin-square code of length
$n$,  with parity check matrix $H$, obtained from $D_1$ by deleting
any $q-n$ columns, where $3 \leq n \leq q$ and $q \geq 3$ is any
prime power \cite{del1}. The intersection array is
$$
\IA = \{n(q-1),(q-n+1)(n-1);1,n(n-1)\}.
$$

\item A $[n=q(q-1)/2,k=n-3,4;2]_q$ code
for $q=2^r \geq 4$ \cite{del1} with
$$
\IA = \{(q-1)n,(q-2)(q+1)(q+2)/4; 1, q(q-1)(q-2)/4\}.
$$
The complementary of this code belongs to the family
$TF1^d$, i.e., it is the projective dual code of a code in the family $TF1$ in \cite{cald}).

\item A $[n = 1 + (q+1)(h-1),k=n-3,4]_q$ code, where
$1 < h < q$ and $h$ divides $q$, for $q=2^r \geq 4$ (the family
$TF2$ in \cite{cald}) with
$$
\IA = \{(q-1)n,(q+1)(h-1)(q-h+1); 1,(h-1)n\}.
$$

\item A $[n = q(q-h+1)/h,k=n-3,4]_q$ code, where
$1 < h < q$ and $h$ divides $q$, for $q=2^r \geq 4$ with
$$
\IA = \{(q-1)n,(q+1)(q-h)(q(h-1)+h)/h^2; 1, q(q-h)(q-h+1)/h^2 \}.
$$

The
complementary of this code belongs to the family $TF2^d$ \cite{cald}.
\end{enumerate}
\end{corollary}

Some of these kinds of codes are self-dual \cite{BRZ10}.

\begin{theorem}[\protect{\cite{BRZ10}}]\label{selfduality}
Let $C'\subset\F_{q^r}^n$ be the lifted code from a Hamming
$[n,n-m,3]_q$ perfect code ${\mathcal H}_m$. Then,
\begin{enumerate}[label=(F.\arabic*),resume*=family]
\item  for any $r$, $C'$ is a CR code and, moreover, $C'$ is self-dual
if and only if ${\mathcal H}_m$ is a ternary $[4,2,3]_3$ Hamming
code.
\end{enumerate}
\end{theorem}

\begin{theorem}[\protect{\cite{BRZ10}}]\label{prop:dual}
Let $\{0,1,\xi_2,
\ldots, \xi_{q-1}\}$ be the field $F_q$, where $q \geq 4$ is any prime power. Let the matrix $D_1$,
\[
D_1~=~\left[
\begin{array}{ccll}
~1~&~1~&~1~&~1\\
~0~&~1~&~\xi_i~&~\xi_j\\
\end{array}
\right],
\]
be a parity check matrix for the code $C$ and a generator matrix for
the code $C^{\perp}$, where  $\xi_i, \xi_j \in \F^*_q$ are two different
elements such that $\xi_i~+~\xi_j~+~1~=~0$. Then
\begin{enumerate}[label=(F.\arabic*),resume*=family]
\item  $C$, as well as $C^{\perp}$, is a linear self-complementary CR $[4,2,3;2]_q$
code with
$$ \IA = \{4\,(q-1), 3\,(q-3); 1, 12\}.$$

\item For the case $q=2^r \geq 4$, these two equivalent codes coincide: $C =
C^{\perp}$, i.e. $C$ is self-dual.
\end{enumerate}
\end{theorem}

\subsection{Binary linear CR codes with $\rho=3$
and $\rho=4$ from bent and AB functions.}\label{ABFunctions}

Let $F$ be any function from $\F_2^m$ to $\F_2^m$. For any $(a,b) \in (\F_2^m)^2$, define the Fourier transform of
$F$ as \EQ\label{eq:7.1}
\mu_F(a,b)~=~\sum_{\bx \in
\F_2^m}\left(-1\right)^{\langle b\cdot F(x)\rangle+\langle a\cdot x\rangle},
\EN
where "$\langle \cdot \rangle$" is the usual inner product on $\F_2^m$.

For even $m$, a function $F$ over $\F_2^m$ is
\textit{bent} if its Fourier transform is $\mu_F(a,b) = \pm
2^{m/2}$, for all $a, b \in \F_2^m$ where $b \neq 0$. For odd $m$,
a function $F$ over $\F_2^m$ is \textit{almost bent} (shortly AB) if its
Fourier transform is $\mu_F(a,b) = \{0, \pm 2^{(m-1)/2}\}$, for
all $a, b \in \F_2^m$ where $b \neq 0$.

Let $F$ be any function from $\F_q$ to $\F_q$ where $q=2^m$, such
that $F(0)=0$. Define
\[
\Omega_m=\left\{
\begin{array}{ccc}
\{2^{m-1}, 2^{m-1} \pm 2^{m/2}\},\;\;& \mbox{if}\;\;m\;\;&\mbox{even},\\
\{2^{m-1}, 2^{m-1} \pm 2^{(m-1)/2}\},\;\;& \mbox{if}\;\;m\;\;&\mbox{odd}.
\end{array}
\right.
\]
For a linear binary code $C$ define the set $W_C$ as a set of all weights
of its nonzero codewords: $W_C = \{\wt(c): ~c \in C,\;c \neq \bf{0}\}$.
Define the matrix $H_F$:
\EQ\label{parity:AB}
H_F = \left[
\begin{array}{ccccccc}
&\,1\,   &\,\alpha\,   &\,\alpha^2\,   &\,\cdots\,&\alpha^{n-1}\,   &\\
&\,F(1)\,&\,F(\alpha)\,&\,F(\alpha^2)\,&\,\cdots\,&F(\alpha^{n-1})\,&
\end{array}
\right],
\EN

The statements of the next theorem can be found in \cite{Car1}

\begin{theorem} [\protect{\cite{Car1}}]
Let $C=C_F$ be the $[n=2^m-1,k,d]$ code defined by the parity check
matrix $H_F$ (\ref{parity:AB}).
\begin{itemize}
\item[(i)] If $m$ is odd, the function $F$ is AB if and only if the
weights of $W_{C^\perp}$ are those in $\Omega_m$.
\item[(ii)] If $m$ is
even, the function $F$ is bent if and only if the weights of
$W_{C^\perp}$ are those in $\Omega_m$.
\item[(iii)] $C_F$ is UP in the wide sense if $|W_{C^\perp}|=3$.
\end{itemize}
\begin{enumerate}[label=(F.\arabic*),resume*=family]
\item $C_F$ is CR, if $F$ is bent ($m$ even).
\item $C_F$ is CR, if $F$ is AB ($m$ odd).
\end{enumerate}
\end{theorem}

Now using the corresponding results from \cite{BRZ14a,BRZ15} we have

\begin{proposition}\label{exten-crc}
Let $C = C_F$ of length $n=2^m-1$, where $m$ is odd, be defined by the parity check matrix
(\ref{parity:AB}) and let $C^\perp$ be its dual code with the set of weights
$W_{C^\perp}$. Then
\begin{itemize}
\item[(i)] $C^*$ is UP if and only if $C$ is UP and
$W_{C^\perp}=\Omega_m$.
\item[(ii)] The code $C^*$  is CR if and only if $C$ is CR
with minimum distance $d \in \{3,5\}$ and $C^*$ is UP.
\end{itemize}
\end{proposition}

\begin{proof}
The first statement comes directly from \cite{BZ}. For the second statement,
the case $d=3$ is known \cite{BRZ15}. Consider the case
$d=5$. Since $C$ is CR, its covering radius is $\rho=3$. Hence $C$ is quasi-perfect.
Now the result follows from Proposition \ref{exten-up}.
\end{proof}

\begin{enumerate}[label=(F.\arabic*),resume*=family]
\item
Taking a power function $F$ in (\ref{parity:AB}) for $m$ odd, we obtain binary primitive cyclic $[n=2^m-1, n-2m, d;\rho]$ codes with generator polynomial $g(x)=m_1(x)m_{\ell}$. In the case $d=5$ they are CR \cite{Car1} with $\rho=3$ and
$$  \IA = \{n,n-1,(n+3)/2;1,2,(n-1)/2\}.$$
We found in the literature the following cases.
%(see references in \cite{Can1}):

\begin{enumerate}[label=(\roman*)]
\item BCH $2$-codes (see \Cref{subsec:preparata}). $\ell = 3$.
\item Gold codes \cite{gold}. $\ell = 2^i + 1$, $\Gcd(i,m)=1$.
\item Kasami codes \cite{kas}. $\ell  = 2^{2i} - 2^i +1$, $\Gcd(i,m)=1$.
\item Welch codes \cite{niho}. $\ell=2^{\frac{m-1}{2}}+3$.
\item Niho codes \cite{niho}. $\ell=2^{2i}+2^i-1$, $4i\equiv -1 \bmod{m}$.
%\item Niho codes \cite{niho}. $\ell=2^{2i}+2^{3i+1}-1$.
\item Inverse \cite{ding}. $\ell=2^{m-1}-1$.
\item Dobbertin codes \cite{dobb}. $\ell=2^{4i}+2^{3i}+2^{2i}-1$, $m=5i$.
\end{enumerate}
\end{enumerate}

All these codes come from  AB
functions.  More new AB functions, which are not power functions
and which provide CR codes can be found in
\cite{Ber,Bud1,Can1}.

\begin{enumerate}[label=(F.\arabic*),resume*=family]
\item The extended codes of all binary codes above are CR with
$\rho=4$ and
$$  \IA = \{n+1,n,n-1,(n+3)/2;1,2,(n-1)/2,n+1\}.$$
\end{enumerate}

%\nocite{*}
%\bibliographystyle{acm}
%\bibliography{survey}
\addcontentsline{toc}{section}{References}
\printbibliography
\end{document}